\documentclass[11pt]{article}
\usepackage{amssymb}
\usepackage{amsmath}
\usepackage{amsthm}
\usepackage{verbatim}
\usepackage{graphicx}
\usepackage{epstopdf}
\usepackage{color}
\usepackage{epsf}
\usepackage[margin=1.161in]{geometry}
\usepackage{subfig}

\usepackage{enumitem}
\setlist{itemsep=0pt, topsep=0pt}

\usepackage{hyperref}

\newtheorem{theorem}{Theorem}[section]
\newtheorem{lemma}[theorem]{Lemma}
\newtheorem{claim}[theorem]{Claim}

\newtheorem{conjecture}[theorem]{Conjecture}

\newtheorem{observation}[theorem]{Observation}
\newtheorem{proposition}[theorem]{Proposition}

\numberwithin{subcase}{case}

\newtheorem{definition}[theorem]{Definition}

\newcommand{\ep}{\epsilon}

\newcommand{\bbE}{{\mathbb E}}

\newcommand{\cA}{{\mathcal A}}
\newcommand{\cB}{{\mathcal B}}

\newcommand{\cH}{{\mathcal H}}
\newcommand{\cM}{{\mathcal M}}

\newcommand{\cS}{{\mathcal S}}
\newcommand{\cT}{{\mathcal T}}
\newcommand{\cW}{{\mathcal W}}
\newcommand{\cX}{{\mathcal X}}

\newcommand{\floor}[1]{\left\lfloor#1\right\rfloor}
\newcommand{\ceiling}[1]{\left\lceil#1\right\rceil}
\newcommand{\con}{\mathrm{con}}

\begin{document}

\title{Monochromatic cycle partitions of graphs with large minimum degree} 

\author{Louis DeBiasio\thanks{Department of Mathematics, Miami University; Oxford, OH 45056\newline {\tt debiasld@miamioh.edu, nelsenll@miamioh.edu}} \thanks{Research supported in part by Simons Foundation Collaboration Grant \# 283194.} \and Luke Nelsen\footnotemark[1]}

\maketitle

\begin{abstract}
Lehel conjectured that in every $2$-coloring of the edges of $K_n$, there is a vertex disjoint red and blue cycle which span $V(K_n)$.  \L uczak, R\"odl, and Szemer\'edi proved Lehel's conjecture for large $n$, Allen gave a different proof for large $n$, and finally Bessy and Thomass\'e gave a proof for all $n$.  

Balogh, Bar\'at, Gerbner, Gy\'arf\'as, and S\'ark\"ozy proposed a significant strengthening of Lehel's conjecture where $K_n$ is replaced by any graph $G$ with $\delta(G)> 3n/4$; if true, this minimum degree condition is essentially best possible. 
We prove that their conjecture holds when $\delta(G)>(3/4+o(1))n$.  Our proof uses Szemer\'edi's regularity lemma along with the absorbing method of R\"odl, Ruci\'nski, and Szemer\'edi by first showing that the graph can be covered with monochromatic subgraphs having certain robust expansion properties.
\end{abstract}


\section{Introduction}

For the purposes of this paper, we consider the empty set, a single vertex, and an edge as cycles on $0$, $1$, and $2$ vertices respectively. By an $r$-coloring of a graph $G$, we mean a partition of its edge set into at most $r$ parts (i.e. exactly $r$ parts, some of which may be empty). Given an $r$-colored graph $G$, a partition of $G$ into monochromatic cycles is a collection of vertex disjoint monochromatic cycles which together span $V(G)$.  We denote a path or cycle on $k$ vertices by $P^k$ and $C^k$ respectively (subscripts will be reserved for colors).

In 1967, Gerencs\'er and Gy\'arf\'as \cite{GG} exactly determined the Ramsey number for all pairs of paths.  In the symmetric case (when the paths have the same length), the result can be stated as follows.

\begin{theorem}[Gerencs\'er, Gy\'arf\'as]\label{GGthm}
Every $2$-coloring of $K_n$ contains a monochromatic $P^k$ with $k> 2n/3$.
\end{theorem}

In 1973, Rosta \cite{Ro} and independently, Faudree and Schelp \cite{FS} exactly determined the Ramsey number for all pairs of cycles, which gave an analog of Theorem \ref{GGthm} for cycles.  Later, this was slightly refined by Faudree, Lesniak, and Schiermeyer \cite{FLS} to give the following best possible result about long monochromatic cycles.

\begin{theorem}[Faudree, Lesniak, Schiermeyer]\label{FLSthm}
For $n\geq 6$, every $2$-coloring of $K_n$ contains a monochromatic $C^k$ with $k\geq 2n/3$.
\end{theorem}

In \cite{GG}, Gerencs\'er and Gy\'arf\'as wrote a small, but historically influential, footnote which contained the seed of a new ``Ramsey-type" partitioning problem.  In the footnote was a simple proof that every $2$-coloring of $K_n$ has a cycle on $n$ vertices which is the union of a blue path and a red path (which in turn contains a monochromatic $P^{\ceiling{n/2}}$).  In a $2$-colored $K_n$, a cycle on $n$ vertices which is the union of a blue path and a red path immediately gives a partition of $K_n$ into two monochromatic paths; from this one can easily deduce that $K_n$ has a partition into a vertex disjoint monochromatic cycle and path of different colors.  Later, Lehel (see \cite{Ayel} and \cite{EGP}) conjectured that  every $2$-coloring of $K_n$ has a partition into a red cycle and blue cycle (note the requirement that the cycles have different colors). 

Lending further support to Lehel's conjecture, Gy\'arf\'as \cite{Gy} proved that in every $2$-coloring of $K_n$ there is a red cycle and a blue cycle which span the vertex set and have at most one common vertex. \L uczak, R\"odl, and Szemer\'edi \cite{LRS} proved Lehel's conjecture for large $n$ and later Allen \cite{Al} gave a different proof of Lehel's conjecture for smaller, but still large $n$.  Finally, Bessy and Thomass\'e \cite{BT} proved Lehel's conjecture for all $n$.

\begin{theorem}[Bessy, Thomass\'e]\label{BTthm}
Every $2$-coloring of $K_n$ has a partition into a red cycle and blue cycle.
\end{theorem}

Schelp \cite{S} raised the general problem of determining whether results such as Theorem \ref{GGthm}, Theorem \ref{FLSthm}, and Theorem \ref{BTthm}, which are about complete graphs, actually hold for graphs with sufficiently large minimum degree.  In particular he conjectured that the conclusion of Theorem \ref{GGthm} still holds if $K_n$ is replaced by any graph $G$ with $\delta(G)>\frac{3n}{4}$.  Gy\'arf\'as and S\"ark\"ozy \cite{GyS} proved that for all $\ep>0$ and sufficiently large $n$, if $G$ is a $2$-colored graph with $\delta(G)\geq (3/4+\ep)n$, then $G$ contains a monochromatic $P^k$ with $k\geq (2/3-\ep)n$.  Then Benevides, \L uczak, Skokan, Scott, and White \cite{BLSSW} proved a Schelp-type analog of Theorem \ref{FLSthm}; that is, for all $\ep>0$ and sufficiently large $n$, if $G$ is a $2$-colored graph with $\delta(G)\geq 3n/4$, then $G$ contains a monochromatic $C^k$ with $k\geq (2/3-\ep) n$ and they conjectured an exact version of this result (see Conjecture 8.3 in \cite{BLSSW}).

Inspired by the above results, Balogh, Barat, Gerbner, Gy\'arf\'as, and S\'ark\"ozy \cite{BBGGS} conjectured the following Schelp-type analog of Theorem \ref{BTthm}.

\begin{conjecture}[Balogh, Bar\'at, Gerbner, Gy\'arf\'as, S\'ark\"ozy] \label{3/4conjecture}
If $G$ is a $2$-colored graph on $n$ vertices with $\delta(G)>\frac{3n}{4}$, then $G$ has a partition into a red cycle and a blue cycle.
\end{conjecture}

They prove that their conjecture nearly holds in an asymptotic sense; that is, for all $\gamma>0$, there exists $n_0$ such that if $G$ is a $2$-colored graph on $n\geq n_0$ vertices with $\delta(G)\geq (\frac{3}{4}+\gamma)n$, then there is a vertex disjoint red cycle and blue cycle spanning at least $(1-\gamma)n$ vertices.

In this paper, we prove that their conjecture holds asymptotically.

\begin{theorem}\label{mainapprox}
For all $\gamma>0$, there exists $n_0$ such that if $G$ is a $2$-colored graph on $n\geq n_0$ vertices with $\delta(G)\geq (\frac{3}{4}+\gamma)n$, then $G$ has a partition into a red cycle and blue cycle.
\end{theorem}

In Section \ref{exact}, we give a small example to show that Conjecture \ref{3/4conjecture} does not hold for all $n$. Despite this, we propose Conjecture \ref{exactconjecture}, a slight strengthening of Conjecture \ref{3/4conjecture} for sufficiently large $n$.

\subsection{Notation}

For a natural number $k$, we write $[k]$ to mean the set $\{1, 2, \dots, k\}$.  Throughout the paper we use ``color 1" and ``red" interchangeably and likewise for ``color 2" and ``blue."  In a $2$-colored graph $G$ with $2$-coloring $E(G)=E_1\cup E_2$, we let $G_i$ be the graph $(V(G), E_i)$ for $i\in [2]$.  We sometimes write $\delta_i(G)$ to mean $\delta(G_i)$.  For subsets $A,B\subseteq V(G)$, we write $\delta(A,B)$ to mean the minimum number of neighbors any vertex in $A$ has in $B$, and we write $E_{G_i}(A,B)$ to mean the set of edges in $G_i$ with an endpoint in $A$ and the other endpoint in $B$.  We also write $e_i(A, B)$ or $e_{G_i}(A, B)$ to mean $|E_{G_i}(A,B)|$.  For a vertex $v\in V(G)$, we write $\deg_i(v)$ in place of $\deg_{G_i}(v)$, and $\deg_i(v, A)$ for $\deg_{G_i}(v, A)$.  Given a graph $G$ and disjoint subsets $X, Y\subseteq V(G)$, we let $G[X,Y]$ be the bipartite subgraph induced by all edges having one endpoint in $X$ and one endpoint in $Y$.  We say a $(U,V)$-bipartite graph is balanced if $|U|=|V|$.

Throughout the paper, we will write $\alpha \ll \beta$ to mean that given $\beta$, we can choose $\alpha$ small enough so that $\alpha$ satisfies all of necessary conditions throughout the proof.  More formally, we can set $\alpha := \min \{f_1(\beta), f_2(\beta), \dots, f_k(\beta)\}$, where each $f_i(\beta)$ corresponds to the maximum value of $\alpha$ allowed so that the corresponding argument in the proof holds.  In order to simplify the presentation, we will not determine these functions explicitly.

\section{Sharpness examples}\label{exact}

\begin{proposition}
There exists a $2$-colored graph $F$ on $9$ vertices with $\delta(G)=7=\frac{3 \cdot 9 + 1}{4}$ such that $F$ does not have a partition into a red cycle and blue cycle.
\end{proposition}

\begin{proof}
\begin{figure}[ht]
\centering
\scalebox{.75}{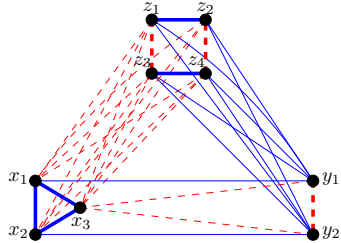}~~~~~~~~~~~~~
\caption[]{A $2$-colored graph (with the red edges shown as dashed lines) $F$ on $9$ vertices with $\delta(F)=7=\frac{3\cdot 9 +1}{4}$ which does not have a partition into a red cycle and blue cycle.}
\label{9vertex}
\end{figure}

(See Figure \ref{9vertex}) Let $F$ be the graph on the vertex set $\{x_1, x_2, x_3, y_1, y_2, z_1, z_2, z_3, z_4\}$ such that the complement of the edge set is $\{z_1z_4, z_3z_2, x_1y_2, x_2y_1\}$.  Color all edges $x_iz_j$ red, all edges $y_iz_j$ blue, all edges $x_ix_j$ blue, $y_1y_2$ red, $x_1y_1$, $x_2y_2$ blue, $x_3y_1$, $x_3y_2$ red, $z_1z_2$, $z_3z_4$ blue, and $z_1z_3$, $z_2z_4$ red.  

The complement of $E(F)$ is a matching and thus $\delta(F)=8-1=\frac{3\cdot 9 +1}{4}$.  Checking cases shows that $F$ does not have a partition into a red cycle and a blue cycle.
\end{proof}

\begin{proposition}
Let $n=4q+r$ with $0\leq r\leq 3$.  For all $r$, there exists a $2$-colored graph $F$ with $\delta(F)=\ceiling{\frac{3n-3}{4}}-1$ such that $F$ does not have a partition into a red cycle and blue cycle.
\end{proposition}

\begin{proof}
\begin{figure}[ht]
\centering

\scalebox{1}{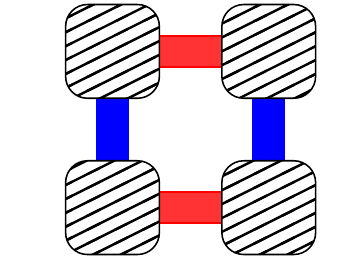}~~~~~~~~~~~~~
\scalebox{1}{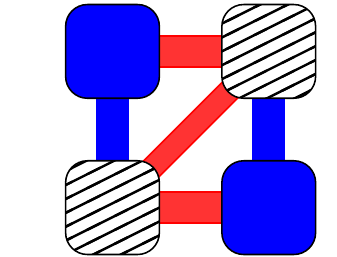}~~~~~~~~~~~~~
\scalebox{1}{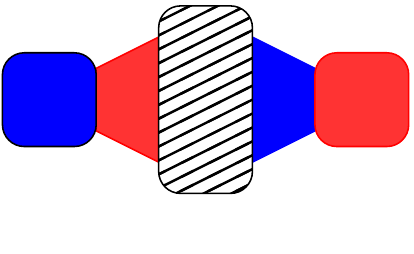}

\caption[]{Three examples of graphs which have minimum degree $\ceiling{\frac{3n-3}{4}}-1$, and do not have a partition into a red cycle and a blue cycle.  The striped lines represent edges whose color has no effect on the example.}
\label{3examples}
\end{figure}

(See Figure \ref{3examples}) Let $n=4q+r$.  Let $\{X_1,X_2,Y_1,Y_2\}$ be a partition of a set with $n$ elements such that (i) $|X_1|+|X_2|+|Y_1|+|Y_2|=n$, (ii) $|X_1|\geq |Y_2| \geq |X_2|,|Y_1|$, and (iii) the difference between the sizes of any pair of sets is at most $1$.  

Let $F_1$ be the graph obtained from the complete graph on $X_1\cup X_2\cup Y_1\cup Y_2$ by deleting all edges between $X_1$ and $Y_2$ and all edges between $X_2$ and $Y_1$.  Now $2$-color the edges of $F_1$ so that all edges between $X_1$ and $X_2$ and between $Y_1$ and $Y_2$ are blue and all edges between $X_1$ and $Y_1$ and between $X_2$ and $Y_2$ are red, and all edges inside $X_1, X_2, Y_1, Y_2$ are arbitrarily colored.  

Let $F_2$ be the graph obtained from the complete graph on $X_1\cup X_2\cup Y_1\cup Y_2$ by deleting all edges between $X_1$ and $Y_2$.  Now $2$-color the edges of $F_2$ so that all edges between $X_1$ and $X_2$ and between $Y_1$ and $Y_2$ are blue, all edges between $X_1$ and $Y_1$ and between $X_2$ and $Y_1\cup Y_2$ are red, all edges inside $X_1$, $Y_2$ are colored blue, and all edges inside $X_2$, $Y_1$ are arbitrarily colored.  

Now let $\{X, Y, Z\}$ be a partition of a set on $n$ elements such that $|X|+|Y|+|Z|=n$,  $||X|-|Y||\leq 1$, and $1\leq |X|+|Y|-|Z|\leq 2$.  

Let $F_3$ be the graph obtained from the complete graph on $X\cup Y\cup Z$ obtained by deleting all edges between $X$ and $Y$.  Now color the edges inside $X$ blue, the edges inside $Y$ red, the edges inside $Z$ arbitrarily, the edges between $X$ and $Z$ red, and the edges between $Y$ and $Z$ blue.

Note that
\[
 \ceiling{\frac{3n-3}{4}}-1 =
  \begin{cases}
   3q-1; & n=4q \\
   3q-1; & n=4q+1 \\
   3q;    & n=4q+2\\ 
   3q+1; & n=4q+3
  \end{cases}
\]

If $n=4q$, then $\delta(F_1)=\delta(F_2)=3q-1$.
If $n=4q+1$, then $\delta(F_1)=\delta(F_2)=\delta(F_3)=3q-1$.
If $n=4q+2$, then $\delta(F_1)=\delta(F_2)=\delta(F_3)=3q$.  
If $n=4q+3$, then $\delta(F_1)=\delta(F_2)=\delta(F_3)=3q+1$.

One can easily check that none of $F_1$, $F_2$, and $F_3$ have a partition into a red cycle and a blue cycle.
\end{proof}

\begin{conjecture}\label{exactconjecture}
There exists $n_0$ such that if $G$ is a $2$-colored graph on $n\geq n_0$ vertices with $\delta(G)\geq \frac{3n-3}{4}$, then $G$ has a partition into a red cycle and a blue cycle.
\end{conjecture}

\textbf{Note added in proof:}  While this paper was under review, Letzter \cite{Let} proved Conjecture \ref{3/4conjecture} for sufficiently large $n$ and gave examples to show that Conjecture \ref{exactconjecture} is false.

\section{Outline of the proof of Theorem \ref{mainapprox}}

As with the proof in \cite{BBGGS} (and many earlier results starting with \cite{Lu} and \cite{LRS}), the idea is to prove that if $G$ is a $2$-colored graph, here with $\delta(G)>3n/4$, then one can find a partition of $G$ into a red matching and a blue matching such that the red matching is contained in a red component and the blue matching is contained in a blue component (i.e. a partition into a red connected matching and a blue connected matching).  Then using Szemer\'edi's regularity method, one can apply this result to a reduced graph to find a vertex disjoint red cycle and blue cycle which span most of the vertices.  In applications of this method where the host graph is complete, it is possible to show that the matchings satisfy certain stronger properties which allow one to insert the remaining vertices in an ad hoc way.  However, since $G$ is not complete, inserting the remaining vertices seems more difficult here.

Our idea is to use the absorbing method of R\"odl, Ruci\'nski, and Szemer\'edi (see \cite{RRSz} and \cite{LSSz}).  However, the Ramsey-type setting introduces some new challenges.  Before applying regularity, we must analyze the structure of the graph and show that $G_1$ and $G_2$ contain contain robust subgraphs, which in this context means they have sufficiently large minimum degree and are highly connected in some sense.  These robust subgraphs can be shown to have certain expansion properties (allowing for a notion of bipartite expansion) and are not sensitive to the deletion of a small number of vertices, which together will allow for absorbing.  Then, regularity is applied so that all clusters lie inside a rough initial partition 
Now proceeding as before, one can find two monochromatic cycles which miss only a small number of vertices and which mostly use edges from $G_1$ and $G_2$.  The absorbing structures allow the leftover vertices to be ``automatically" inserted into the cycles, thus completing the monochromatic cycle partition.

In Section \ref{sec:prelim} we introduce some preliminary lemmas, in Section \ref{sec:conabs} we prove that robust components have the connecting/absorbing property (the results of this section are independent of the edge-colored setting of this paper and can have other applications), in Section \ref{sec:structure} we prove structural results regarding robust components, in Section \ref{sec:proof} we prove a result about connected matchings and complete the proof, and finally in Section \ref{sec:conclusion} we make some concluding remarks.

\section{Preliminary material}\label{sec:prelim}

\begin{lemma}[Chv\'atal \cite{Ch}]\label{ChvatalDegree}
\begin{enumerate}
\item 
Let $G$ be a graph on $n\geq 3$ vertices and let $d_1\leq d_2\leq \dots\leq d_n$ be the degree sequence of $G$.  If for all $1\leq i<n/2$ we have $d_i\geq i+1$ or $d_{n-i}\geq n-i$, then $G$ contains a Hamiltonian cycle.
\item Let $G$ be a balanced $(U,V)$-bipartite graph on $2n \ge 4$ vertices.  Suppose that $\deg(u_1)\leq \deg(u_2)\leq \dots \leq \deg(u_n)$ and $\deg(v_1)\leq \deg(v_2)\leq \dots \leq \deg(v_n)$.  If $\deg(u_i)>i$ or $\deg(v_{n-i})> n-i$ for all $1\leq i<n$, then $G$ has a Hamiltonian cycle.
\end{enumerate}
\end{lemma}

\subsection{Definitions and observations}

\begin{definition}[$\alpha$-sparse cut]
Let $0<\alpha$ and let $G$ be a graph on $n\geq n_0$ vertices. For disjoint $X,Y\subseteq V(G)$, we say $(X,Y)$ is an $\alpha$-sparse pair if $e(X,Y)<\alpha|X||Y|$. We say $G$ has an $\alpha$-sparse cut if there exists $X\subseteq V(G)$ such that $(X,V(G)\setminus X)$ is an $\alpha$-sparse pair.
\end{definition}

\begin{definition}[$(\eta, \alpha)$-robust]
Let $0<\alpha,\eta$ and let $G$ be a graph on $n$ vertices.  A subgraph $H\subseteq G$ is $(\eta,\alpha)$-robust if $\delta(H)\geq \eta n$ and $H$ has no $\alpha$-sparse cut.  We say $X\subseteq V(G)$ is $(\eta, \alpha)$-robust if $G[X]$ is $(\eta, \alpha)$-robust.

We say that $G$ has an $(\eta, \alpha)$-robust partition if there exists a partition $\{V_1, \dots, V_k\}$ of $V(G)$ such that $G[V_i]$ is $(\eta, \alpha)$-robust for all $i\in [k]$.
\end{definition}

The following simple observation basically says that if the minimum degree is at least $\eta n$, then any $\alpha$-sparse cut $\{X,Y\}$ must have both $|X|$ and $|Y|$ be sufficiently large.  

\begin{observation}\label{simplerobust}
Let $0<\alpha\leq \eta/2$, let $G$ be a graph on $n$ vertices, and let $\{X_1,X_2\}$ be a partition of $V(G)$ with $|X_1|\leq |X_2|$.  If $\delta(G)\geq \eta n$ and $|X_1|\leq \eta n/2$, then $e(X_1, X_2)\geq \frac{\eta}{2}|X_1||X_2|\geq \alpha|X_1||X_2|$.
\end{observation}

\begin{proof}
Suppose $\delta(G)\geq \eta n$ and $|X_1|\leq \eta n/2$.  By the minimum degree condition, we have $e(X_1, X_2)\geq (\eta n-|X_1|)|X_1|\geq \frac{\eta}{2}n|X_1|\geq \frac{\eta}{2}|X_1||X_2|\geq \alpha|X_1||X_2|$.
\end{proof}

The following two observations (using slightly different language) are proved in \cite{CFS}.

\begin{observation}[\cite{CFS} Lemma 6.1]\label{cleanedsparsecut}
Let $0<\alpha\leq \eta/2$ and let $G$ be a graph on $n$ vertices with $\delta(G)\geq \eta n$.  If $G$ has an $\alpha$-sparse cut, then there exists another partition $\{X_1, X_2\}$ of $V(G)$ such that $e(X_1, X_2)\leq \alpha n^2$ and $\delta(G[X_i])\geq (\eta-\frac{5\alpha}{\eta})|X_i|$ for $i\in [2]$. 

\end{observation}

\begin{observation}[\cite{CFS} Lemma 6.2]\label{robustpartition}
Let $0<\alpha\leq \eta^3/80$ and let $G$ be a graph on $n$ vertices. If $\delta(G)\geq \eta n$, then there exists a partition $\{V_1,\dots, V_k\}$ of $V(G)$ such that for all $i\in [k]$ we have that $|V_i|\geq \eta n/2$ (which implies $k\leq \frac{2}{\eta}$), and $G[V_i]$ has no $\alpha$-sparse cuts, $\delta(G[V_i])\geq \eta |V_i|/2$.
In particular, $G$ has an $(\eta^2/4, \alpha)$-robust partition.
\end{observation}

\begin{observation}\label{slicerobust}
Let $0<\alpha\leq \eta/2$ and let $G$ be a graph on $n$ vertices. If $G$ is $(\eta, \alpha)$-robust and $Z\subseteq V(G)$ with $|Z|\leq \frac{\alpha\eta}{8} n$, then $G-Z$ is $(\eta/2, \alpha/2)$-robust.
\end{observation}

\begin{proof}
The minimum degree condition follows immediately since $|Z|\leq\frac{\alpha\eta}{8} n\leq  \eta n/2$.  Suppose there is a partition $\{X_1, X_2\}$ of $V(G)\setminus Z$ with $|X_1|\leq |X_2|$ such that $e(X_1, X_2)<\frac{\alpha}{2}|X_1||X_2|$.  Note that this implies $|X_1|\geq \eta n/4$ as otherwise by Observation \ref{simplerobust} we would have $e(X_1, X_2)\geq \frac{\alpha}{2}|X_1||X_2|$.  So we have 
\begin{align*}
e(X_1\cup Z, X_2)= e(X_1, X_2)+e(Z, X_2)< \frac{\alpha}{2}|X_1||X_2|+ \frac{\alpha\eta}{8} n|X_2|&=\alpha(\frac{|X_1|}{2}+\frac{\eta/4}{2} n)|X_2|\\
&\leq \alpha|X_1\cup Z||X_2|.
\end{align*}
which implies that $(X_1\cup Z, X_2)$ is an $\alpha$-sparse cut in $G$, contradicting the original assumption. 
\end{proof}

\begin{definition}[$\eta'$-maximal extension]
Let $0<\alpha,\eta',\eta$, let $G$ be a graph on $n$ vertices, and let $H^0\subseteq G$ such that $H^0$ is $(\eta, \alpha)$-robust.  Consider the following process: for $i\geq 1$, if there exists $v_i\in V(G)\setminus V(H^{i-1})$ with $\deg(v_i, H^{i-1})\geq \eta'n$, let $H^i:=G[V(H^{i-1})\cup \{v_i\}]$; if not, set $k:=i-1$.  We call $H^k$ an $\eta'$-maximal extension of $H^0$.
\end{definition}

\begin{observation}\label{augmentrobust}
Let $0<\alpha\leq \eta'/2\leq \eta/2$ and let $G$ be a graph on $n$ vertices.  If there exists $H^0\subseteq G$ such that $H^0$ is $(\eta, \alpha)$-robust and $H^k$ is an $\eta'$-maximal extension of $H^0$, then $H^k$ is $(\eta', \alpha\eta'\tau)$-robust where $\tau:=\frac{|V(H^0)|}{|V(H^k)|}$.
\end{observation}

\begin{proof}
The minimum degree condition follows immediately from the definition.  

Set $n_0:=|V(H^0)|$, in which case we can write $\tau=\frac{n_0}{n_0+k}\geq \frac{\eta n}{n}=\eta$.  Let $\{Y_1, Y_2\}$ be a partition of $V(H^k)$ 
such that 
\begin{equation}\label{aeb}
e(Y_1, Y_2)<\alpha\eta'\tau|Y_1||Y_2|\leq \alpha\eta'\tau n^2/4.
\end{equation}
For all $j\in [2]$, set $X_j:=Y_j\cap V(H^0)$; without loss of generality, suppose $|X_1|\leq |X_2|$.  
Since $\delta(H^k)\geq \eta'n$, Observation \ref{simplerobust} implies that $|Y_1|\geq \eta'n/2$.

If $|X_1|\leq \eta'n/2$, then each vertex in $X_1$ has at least $\eta'n/2$ neighbors in $X_2$ and the first $\ceiling{\eta'n/2}-|X_1|$ vertices which are added to $Y_1$ in the process each have at least $\eta'n/2$ neighbors in $Y_2$. So 
\begin{align*}
e(Y_1, Y_2)\geq e(X_1, X_2)+(\eta'n/2-|X_1|)\eta'n/2\geq |X_1|\eta'n/2+(\eta'n/2-|X_1|)\eta'n/2
=\eta'^2n^2/4,
\end{align*}
contradicting \eqref{aeb}.

So suppose $|X_1|> \eta'n/2$, which implies that 
\begin{equation}\label{X1X2size}
4|X_1||X_2|>4\frac{\eta'n}{2}(n_0-\frac{\eta'n}{2})\geq 2\eta'(1-\frac{\eta'n}{2n_0})nn_0\geq 2\eta'(1-\frac{\eta'}{2\eta})nn_0\geq \eta'nn_0. 
\end{equation}

Since $H^0$ is $(\eta, \alpha)$-robust, we have $e(X_1, X_2)\geq \alpha |X_1||X_2|$ and by \eqref{X1X2size} and the fact that $|Y_1||Y_2|\leq \left(\frac{n_0+k}{2}\right)^2$, we have 
\begin{align*}
e(Y_1, Y_2)\geq e(X_1, X_2)\geq \alpha\frac{|X_1|}{|Y_1|}\frac{|X_2|}{|Y_2|}|Y_1||Y_2|\geq \alpha \frac{4|X_1||X_2|}{(n_0+k)^2}|Y_1||Y_2|&\geq \alpha\frac{\eta'nn_0}{(n_0+k)^2}|Y_1||Y_2|\\
&\geq \alpha\eta'\tau|Y_1||Y_2|,
\end{align*}
which contradicts \eqref{aeb}.
\end{proof}

\begin{definition}[$\beta$-near-bipartite]
Let $0<\beta,\eta$ and let $G$ be a graph on $n$ vertices.  We say $G$ is $\beta$-near-bipartite if there exists $X\subseteq V(G)$ such that $e(X)<\beta n^2$ and $e(V(G)\setminus X)<\beta n^2$.  If in addition to this we have $\delta(X, V(G)\setminus X)\geq \eta n$ and $\delta(V(G)\setminus X, X)\geq \eta n$, then we say $G$ has a $(\eta, \beta)$-bipartition.
\end{definition}

\begin{observation}\label{bipartiterobust} 
Let $0<\alpha\leq \eta/2$ and $0<\beta\leq \alpha^{3/2}/9$, and let $G$ be a graph on $n$ vertices.  If $G$ is $(\eta, \alpha)$-robust and $\beta^2$-near-bipartite, then $G$ has an $(\eta/2, \beta)$-bipartition $\{S_1, S_2\}$ such that $H:=G[S_1,S_2]$ is $(\eta/2, \alpha/2)$-robust.  
\end{observation}

\begin{proof}
Let $\{S_1', S_2'\}$ be a partition of $G$ such that $e(S_i')<\beta^2 n^2$ for all $i\in [2]$.  For $i\in [2]$, let $T_i'=\{v\in S_i': \deg(v, S_{3-i}')<3\eta n/4\}$.  Since $$\frac{1}{2}|T_i'|\eta n/4 \leq e(S_i') < \beta^2 n^2,$$ we have $|T_i'|<\frac{8\beta^2}{\eta}n$ for each $i\in [2]$.   Let $\{T_1, T_2\}$ be a partition of $T_1'\cup T_2'$ which maximizes $e((S_1'\setminus T_1')\cup T_1, (S_2'\setminus T_2')\cup T_2)$ and set $S_i:=(S_i'\setminus T_i')\cup T_i$ for $i\in [2]$.  For all $v\in S_i\setminus T_i$, we have $$\deg(v, S_{3-i})\geq 3\eta n/4-|T_{3-i}'|\geq  3\eta n/4-\frac{8\beta^2}{\eta}n\geq  \eta n/2.$$  For all $v\in T_i$, we have $\deg(v, S_{3-i})\geq \eta n/2$, as otherwise we could move $v$ to $T_{3-i}$ to increase the number of crossing edges, contradicting the choice of $\{T_1, T_2\}$.  For each $i\in [2]$, we also have $$e(S_i)\leq e(S_i')+e(T_{i}, S_i)\leq \beta^2 n^2+\frac{8\beta^2}{\eta}n|S_i|\leq \frac{9\beta^2}{\eta}n^2\leq \beta n^2,$$ which completes the proof that $\{S_1, S_2\}$ is an $(\eta/2, \beta)$-bipartition.

To see that $H=G[S_1,S_2]$ is $(\eta/2, \alpha/2)$-robust, first note that the degree condition follows from the definition of $(\eta/2, \beta)$-bipartition.  Let $\{X_1, X_2\}$ be a partition of $V(H)$  with $|X_1|\leq |X_2|$.  If $|X_1|\leq \eta n/4$, then by the degree condition and Observation \ref{simplerobust} we have $e_H(X_1, X_2)\geq \frac{\alpha}{2} |X_1||X_2|$.  So suppose $|X_1|>\eta n/4$.  Since $G$ has no $\alpha$-sparse cuts, we have $e_G(X_1, X_2)\geq \alpha|X_1||X_2|$ and thus 
$$e_H(X_1, X_2)= e_G(X_1, X_2)-e(S_1)-e(S_2)\geq \alpha|X_1||X_2|-\frac{18\beta^2}{\eta}n^2\geq \frac{\alpha}{2}|X_1||X_2|.$$ 
Where the last inequality follows by $\alpha\leq \eta/2$ and $\beta\leq \alpha^{3/2}/9$ and $|X_1||X_2|\geq \eta n/4(1-\eta/4)n\geq \eta n^2/8$.
\end{proof}

\subsection{Probability}

It will be helpful to have the following version of Markov's inequality.

\begin{lemma}[Markov]\label{markov}
Let $S$ be a finite multiset of non-negative real numbers.  Denote the sum of the elements of $S$ by $\Sigma$ and their average value by $\mu$.  For $a>0$, set $S_{\leq a}=\{i\in S: i\leq a\}$ and $S_{\geq a}=\{i\in S: i\geq a\}$.

\begin{enumerate}
\item $|S_{\geq a}|\leq \frac{\mu}{a}|S|=\frac{\Sigma}{a}$.

\item If $a \le \mu< \max \{S\}\leq b$, then $|S_{\leq a}|\leq \frac{b-\mu}{b-a}|S|=\frac{b|S|-\Sigma}{b-a}$.
\end{enumerate}
\end{lemma}

\begin{proof}
Both parts follow from the fact that $a|S_{\geq a}|\leq \Sigma\leq a|S_{\leq a}|+b(|S|-|S_{\leq a}|)$.
\end{proof}

\begin{lemma}[Chernoff]
Let $X$ be a binomial or hypergeometric random variable.  Then for all $0<\ep<3/2$,
\begin{equation*}
Pr(|X-\mathbb{E}X|\geq \ep \mathbb{E}X)\leq 2\exp\left(-\frac{\ep^2}{3}\mathbb{E}X\right).
\end{equation*}
\end{lemma}

\subsection{Regularity}

Implicit in the proof of the regularity lemma \cite{Sz} is the fact that one can start with an arbitrary initial partition of the vertex set (into parts that are not too small) and obtain an $\ep$-regular partition which has the property that all parts are subsets of the initial partition.\footnote{The initial partition consisting of $\ell$ parts with $n\gg 1/\ell$ is refined over and over until the $\ep$-regular partition is obtained.}  Below is the standard degree form for the $2$-colored regularity lemma (see \cite{KS}) with this fact made explicit.  We call $\{E_1, E_2\}$ a $2$-multicoloring of $G$ if $E_1\cup E_2=E(G)$ (i.e. we allow for $E_1\cap E_2\neq \emptyset$).

\begin{lemma}[2-colored regularity lemma -- degree form]\label{2colordegreeform}
Let $G$ be a 2-colored graph on $n$ vertices, let $0<\rho<1/2$, and let $\{Q_1,\dots,Q_\ell\}$ be a partition of $V(G)$ with $|Q_i|\geq \rho n$ for all $i\in [\ell]$.  For all $0< \ep\ll \rho$ and  $m\geq \ell$, there exists an $M = M(\ep, m)$ such that if $d \in [0,1]$ is any real number, then there is $m \le k \le M$, a partition $\{V_0, V_1,\dots, V_k\}$ of the vertex set $V$ and a subgraph $G' \subseteq G$ with the following properties:
\begin{enumerate}
\item for all $i\in [k]$, there exists $j\in [\ell]$ such that $V_i\subseteq Q_j$,
\item $|V_0| \le \ep n$,
\item all clusters $V_1, \dots, V_k$ are of the same size $|V_1|\le\lceil\ep n\rceil$,
\item $d_{G'}(v)>d_G(v)-(2d+\ep)n$ for all $v\in V$,
\item $e(G'[V_i]) = 0$ for all $i \in [k]$,
\item for all $1 \le i < j \le k$, the pair $(V_i,V_j)$ is $\ep$-regular in $G_1'$ with a density either 0 or greater than $d$ and $\ep$-regular in $G_2'$ with a density either 0 or greater than $d$, where $E(G') = E(G_1') \cup E(G_2')$ is the induced 2-coloring of $G'$.
\end{enumerate}
\end{lemma}

\begin{definition}[$(\ep,d)$-reduced graph]\label{def:reduced}
Given a graph $G$, an initial partition $\{Q_1, \dots, Q_\ell\}$, and a partition $\{V_0, V_1, \dots, V_k\}$ satisfying conditions (i)-(vi) of Lemma \ref{2colordegreeform}, we define the $(\ep, d)$-reduced graph of $G$ to be the graph $\Gamma$ on vertex set $\{V_1, \dots, V_k\}$ such that $V_iV_j$ is an edge of $\Gamma$ if $G'[V_i, V_j]$ has density at least $2d$.  For each $V_iV_j\in E(\Gamma)$, we assign color $1$ if $G'_1[V_i, V_j]$ has density at least $d$ and color $2$ if $G'_2[V_i, V_j]$ has density at least $d$ (note that since the total density is at least $2d$ every edge must receive a color, but it need not be unique). 
\end{definition}

The fact that edges can receive two colors won't bother us as later on we will find a matching in the reduced graph and at that point (but only at that point) we can choose an arbitrary color for the edge.

The following is a well known consequence of the regularity lemma (see Proposition 42 in \cite{KO}).

\begin{lemma}\label{reduceddegree}
Let $0<2\ep\leq d\leq c/2$ and let $G$ be a graph on $n$ vertices with $\delta(G)\geq cn$. If $\Gamma$ is a $(\ep, d)$-reduced graph of $G$ obtained by applying Lemma \ref{2colordegreeform}, then $\delta(\Gamma)\geq (c-2d)k$.
\end{lemma}

We now prove that the reduced counterparts of robust components remain connected in the reduced graph. Note that it is possible to prove that robustness (with slightly relaxed parameters) is inherited by the reduced graph, but for our purposes, this is not needed.

\begin{lemma}\label{connectedcomponent}
Let $0<\ep, d, \eta, \alpha, \rho$ be chosen so that $\ep\ll \rho$ and $4d+2\ep<\alpha \eta$.  Let $G$ be a $2$-colored graph and suppose there exists $X\subseteq V(G)$ such that $X$ is $(\eta, \alpha)$-robust in $G_i$ for some $i\in [2]$.  Suppose $\{Q_1, \dots, Q_\ell\}$ is a partition of $V(G)$ which refines $\{X, V(G)\setminus X\}$ and satisfies $|Q_i|\geq \rho n$ for all $i\in [\ell]$.  If $\Gamma$ is the $(\ep, d)$-reduced graph of $G$ respecting the given partition, then the reduced graph of color $i$ induced by the clusters contained in $X$ is connected.

\end{lemma}

\begin{proof}
Without loss of generality, suppose $X\subseteq V(G)$ is $(\eta, \alpha)$-robust in $G_1$.  After applying Lemma \ref{2colordegreeform} to $G$ with initial partition $\{Q_1, \dots, Q_\ell\}$, let $\cX$ be the set of clusters which are subsets of $X$ and suppose that $\Gamma_1[\cX]$ is not connected.  Let $\cA$ be the smallest component of $\Gamma_1[\cX]$ and let $\cB=\cX-\cA$.  Let $A=\bigcup_{V\in V(\cA)}V$ and $B=\bigcup_{V\in V(\cB)}V$; note that $|B|\geq |X|/2\geq \eta n /2$.  We have, by property (iv) of Lemma \ref{2colordegreeform},
$$e_{G_1}(A, B)< |A|(2d+\ep)n<\alpha |A||B|$$
contradicting the fact that $G_1[X]$ is $(\eta, \alpha)$-robust.
\end{proof}

\section{Connecting and Absorbing}\label{sec:conabs}

In this section, we prove that $(\eta, \alpha)$-robust graphs $G$ have the property that between any pair of vertices there are many short paths, and either every vertex is in many short odd cycles or $G$ is close to bipartite and pairs of vertices from opposite sides of the bipartition are in many short even cycles.  Together these properties will essentially allow us to say that in an $(\eta, \alpha)$-robust graph, a nearly spanning cycle is essentially as good as a spanning cycle.  To put this into an existing context, say that $G$ is a \emph{$(\nu, \tau)$-robust-expander} if for all $S\subseteq V(G)$ with $\tau n\leq |S|\leq (1-\tau)n$, $|\{v: \deg(v, S)\geq \nu n\}|\geq |S|+\nu n$.  K\"uhn, Osthus, and Treglown \cite {KOT} proved (stated here for undirected graphs) that for $0<\frac{1}{n_0}\ll \nu\leq\tau\ll\eta$, if $G$ is a graph on $n\geq n_0$ vertices such that $\delta(G)\geq \eta n$ and $G$ is a $(\nu, \tau)$-robust-expander, then $G$ has a hamiltonian cycle.  The results of this section show that properties weaker than ``robust-expansion'' imply absorption and thus reduces the problem of finding a spanning cycle in such a graph to finding a nearly spanning cycle.

\begin{definition}[Neighborhood cascade]
Let $G$ be a graph and let $x\in V(G)$.  A $(k,\alpha)$-neighborhood cascade of $x$ is a collection of disjoint sets $\{X_1, \dots, X_k\}$ such that $X_1=N(x)$ and for all $1\leq i\leq k-1$ we have $\delta(X_{i+1}, X_{i})\geq \alpha n/k$.  If $V(G)=\{x\}\cup\bigcup_{1\leq i\leq k}X_i$, then we say that the neighborhood cascade is spanning.
\end{definition}

\begin{lemma}\label{neighborcascade}
Let $0< \alpha\leq \eta/8$ and let $G$ be a graph on $n$ vertices. If $G$ is $(\eta, \alpha)$-robust, then for all $x\in V(G)$ there exists a spanning $(k, \alpha^2)$-neighborhood cascade of $x$ with $1\leq k\leq \floor{1/\alpha^2}-1$.
\end{lemma}

\begin{proof}
Let $x\in V(G)$ and set $\widetilde{X}_1=N(x)$.  If $|\widetilde{X}_1|=n-1$, then we are done; so suppose not.  For $i\geq 1$, set $X_{\leq i}= \{x\} \cup \bigcup_{1\leq j\leq i} \widetilde{X}_j$ and $\widetilde{X}_{i+1}=\{v\in V(G)\setminus X_{\leq i} :\deg(v, X_{\leq i})\geq \alpha^2 n\}$.

For any $i\geq 1$, if $|X_{\leq i}|<(1-\eta/2)n$, then $$e(X_{\leq i}, V(G)\setminus X_{\leq i})\geq \alpha |X_{\leq i}||V(G)\setminus X_{\leq i}|\geq \alpha \eta/2(1-\eta/2)n^2\geq \alpha\eta n^2/4$$ and thus $|\widetilde{X}_{i+1}|\geq \frac{\alpha\eta n^2/4-\alpha^2n^2}{n}\geq \alpha^2n.$

This implies that $|X_{\leq i}|\geq (1-\eta/2)n$ for some integer $i\leq \frac{1-\eta/2}{\alpha^2}\leq \floor{1/\alpha^2}-2$, which implies $X_{\leq i+1}=V(G)$ since $\delta(G)\geq \eta n$; note that $i+1\leq 1/\alpha^2-1$.  Let $k_0$ be minimum such that $X_{\leq k_0}=V(G)$ and note that as stated above
$1\leq k_0\leq 1/\alpha^2-1.$

We will now consider each $2\leq i\leq k_0$ one by one and update the sets $\widetilde{X}_2,\dots, \widetilde{X}_i$ each time.  We proceed from $i=2$ to $i=k_0$.  Let $h\leq i$ be the number of sets in $\{\widetilde{X}_1, \dots, \widetilde{X}_i\}$ which are non-empty and for all $1\leq j\leq i-1$, let $$\widetilde{X}_i(j)=\{v\in \widetilde{X}_i: j \text{ is minimum such that } \deg(v, \widetilde{X}_j)\geq \alpha^2 n/h\}.$$
Note that by the definition of $\widetilde{X}_i$, the collection $\{\widetilde{X}_i(1),\dots, \widetilde{X}_i(i-1)\}$ forms a partition of $\widetilde{X}_i$ (where some of the $\widetilde{X}_i(j)$'s may be empty).  Now for all $2\leq j\leq i-1$, set $\widetilde{X}_j:=\widetilde{X}_j\cup \widetilde{X}_i(j-1)$ and $\widetilde{X}_i:=\widetilde{X}_i(i-1)$.  At the end of this process let $k\leq k_0$ be maximal such that ${X}_k\neq \emptyset$. For all $1\leq i\leq k-1$ we have $\delta({X}_{i+1}, {X}_{i})\geq \alpha^2 n/k$, as desired.
\end{proof}

\subsection{Connecting}\label{section:connecting}

\begin{definition}[$(k,\alpha)$-connecting property]
Let $G$ be a graph on $n$ vertices.  For $x,y\in V(G)$, let $\con_i(x,y)$ be the set of $x,y$-paths having $i$ internal vertices.  We say $G$ has the $(k, \alpha)$-connecting property if for all $x,y\in V(G)$, there exists $1\leq i\leq k$ such that $|\con_i(x,y)|\geq (\alpha n)^i$.
\end{definition}

The following lemma essentially says that graphs are robust if and only if they have the connecting property.  

\begin{lemma}[Connecting Lemma]\label{connectingrobust}
Let $0< \frac{1}{n_0}\ll\alpha, \eta$ and let $G$ be a graph on $n\geq n_0$ vertices. 
\begin{enumerate}
\item If $\eta\geq 2\alpha^2$ and $G$ is $(\eta, \alpha)$-robust, then $G$ has the $(\frac{1}{\alpha^2}, \alpha^4)$-connecting property.
\item If $\delta(G)\geq \eta n$ and $G$ has the $(\frac{1}{\sqrt{\alpha}}-1,\alpha)$-connecting property, then $G$ is $(\eta, \alpha^{k+1})$-robust.
\end{enumerate}
\end{lemma}

\begin{proof}
(i) First suppose $G$ is $(\eta, \alpha)$-robust and let $x,y\in V(G)$.  By Lemma \ref{neighborcascade}, there exists a spanning $(k, \alpha^2)$-neighborhood cascade of $x$, say $\{X_1,\dots, X_k\}$ with 
\begin{equation}\label{cascadeproperty}
1\leq k\leq \frac{1}{\alpha^2}-1 ~~\text{ and }~~ \delta(X_{i+1}, X_{i})\geq \alpha^2 n/k ~\text{ for all } 1\leq i\leq k-1.
\end{equation}
Since $\delta(G)\geq \eta n$, there exists some $1\leq j\leq k$ such that $|N(y)\cap X_j|\geq (\eta n-1)/k$.  By \eqref{cascadeproperty} we have $$|\con_j(x,y)|\geq \frac{\eta n-1}{k}\cdot \left(\frac{\alpha^2 n}{k}\right)^{j-1}\geq (\alpha^4n)^j.$$

(ii) Suppose $\delta(G)\geq \eta n$ and for all $x,y\in V(G)$, there exists some $1\leq k\leq \frac{1}{\sqrt{\alpha}}-1$ such that there are at least $(\alpha n)^k$ $x,y$-paths having $k$ internal vertices. Suppose for a contradiction that $G$ is not $(\eta, \alpha^{k+1})$-robust.  Since $\delta(G)\geq \eta n$, this implies $G$ has an $\alpha^{k+1}$ sparse cut $(X,Y)$.  

Set $k':= \floor{\frac{1}{\sqrt{\alpha}}}-1$.  Note that there exists some $1\leq k\leq k'$ such that at least $|X||Y|/k'$ of the pairs $(x,y)$ with $x\in X$ and $y\in Y$ have at least $(\alpha n)^k$ $x,y$-paths having $k$ internal vertices.
Let $\mathcal{P}_{k}(X,Y)$ be the set of all paths having $k$ internal vertices with the first vertex in $X$ and the last vertex in $Y$.  So 
\begin{equation}\label{PXY}
|\mathcal{P}_k(X,Y)|\geq \frac{1}{k'}|X||Y|(\alpha n)^k.
\end{equation}  

Each path in $\mathcal{P}_k(X,Y)$ uses at least one edge from $E(X,Y)$, so for each $uv\in E(X,Y)$ with $u\in X$ and $v\in Y$, there are at most $(k+1)(n-2)(n-3)\cdots(n-2-k+1)<(k+1)n^k$ paths $P\in \mathcal{P}_k(X,Y)$ in which $uv$ is the first edge from $E(X,Y)$ to appear on $P$ (as such a path has $k+1$ edges and $uv$ can appear in any of those $k+1$ positions).  Thus $$|\mathcal{P}_k(X,Y)|< e(X,Y)(k+1)n^k< \alpha^{k+1}|X||Y|(k+1)n^k\leq \frac{1}{k+1}|X||Y|(\alpha n)^k,$$
contradicting \eqref{PXY}. 
\end{proof}

\subsection{Absorbing}\label{section:absorbing}

\begin{definition}[Absorbing Property]
Let $G$ be a graph on $n$ vertices.
\begin{enumerate}
\item We say $G$ has the $(2\ell,\alpha)$-vertex-absorbing property if for all $v\in V(G)$ there exists $2i\leq 2\ell$ such that $v$ is contained in at least $(\alpha n)^{2i}$ cycles of length $2i+1$.

\item We say that $G$ has the $(4\ell,\alpha)$-pair-absorbing property if $G$ contains a spanning bipartite subgraph $H=G[X,Y]$ such that for all $x\in X$ and $y\in Y$ there exists $4i\leq 4\ell$ such that there are at least $(\alpha n)^{4i}$ cycles of length $4i+2$ in $H$ containing $x$ and $y$ in which $x$ and $y$ are at distance $2i+1$ on the cycle (in other words there are $2i$ internal vertices between $x$ and $y$ in either direction on the cycle).
\end{enumerate}
\end{definition}

\begin{lemma}[Absorbing Lemma]\label{lemma:absorbing}
Let $\frac{1}{n_0}\ll \alpha\ll \eta$, set $\rho:=\alpha^{32/\alpha^2}$, and suppose $G$ is an $(\eta, \alpha)$-robust graph on $n\geq n_0$ vertices.  
\begin{enumerate}
\item If $G$ is not $\alpha^4$-near-bipartite, then there exists a path $P^*$ of length at most $\rho n$, such that for all $W\subseteq V(G)\setminus V(P^*)$ with $|W|\leq \rho^3 n$, the subgraph $G[V(P^*)\cup W]$ contains a spanning path having the same endpoints as $P^*$.  

\item If $G$ is $\alpha^4$-near-bipartite, then $G$ has a spanning bipartite subgraph $H=G[X,Y]$ such that $H$ is $(\eta/2,\alpha/2)$-robust and contains a path $P^*$ of length at most $\rho n$, such that for all $W\subseteq V(G)\setminus V(P^*)$ with $|W\cap X|=|W\cap Y|\leq \rho^3 n$, the subgraph $G[V(P^*)\cup W]$ contains a spanning path having the same endpoints as $P^*$.
\end{enumerate}
\end{lemma}

To prove Lemma \ref{lemma:absorbing}, we need the following two preliminary results.  Proposition \ref{manygadgets} is specific to this application and Proposition \ref{generalabsorbing} is the general machinery.  While many recent papers have used the absorbing lemma (notably \cite{RRSz} and \cite{LSSz}), we still need to provide a proof of Proposition \ref{generalabsorbing} here, as this is the only application (to our knowledge) where the absorbing sets have different sizes.  This issue requires a bit more care, although the idea is the same.

\begin{proposition}\label{manygadgets}
Let $0<\frac{1}{n_0}\ll\alpha\ll \eta$ and let $G$ be a graph on $n\geq n_0$ vertices.  If $G$ is $(\eta, \alpha)$-robust, then either $G$ has the $(2\floor{1/\alpha^2}, \alpha^4)$-vertex-absorbing property or $G$ is $\alpha^4$-near bipartite and $G$ has the $(4\ell, (\alpha/4)^4)$-pair-absorbing property for some integer $\ell$ with $\ell\leq 2/\alpha^2$.
\end{proposition}

\begin{proof}
Suppose $G$ is $(\eta, \alpha)$-robust.  
First assume that $G$ is not $\alpha^4$-near bipartite.
Let $x\in V(G)$; by Lemma \ref{neighborcascade}, there exists a spanning $(k, \alpha^2)$-neighborhood cascade of $x$, say $\{X_1, \dots, X_k\}$ with $k\leq \floor{\frac{1}{\alpha^2}}-1$.  Let $Y_1=\bigcup_{1\leq 2j+1\leq k}X_{2j+1}$ and $Y_2=\bigcup_{2\leq 2j\leq k}X_{2j}$ (i.e. $Y_1$ is the union of the odd indexed sets and $Y_2$ is the union of the even indexed sets). Since $G$ is not $\alpha^4$-near-bipartite, we may suppose without loss of generality that $e(Y_1)\geq \alpha^4 n^2$ or $e(Y_1 \cup \{x\})\geq \alpha^4 n^2$; in either case, we have $e(Y_1)\geq \alpha^4 n^2 -n$.

By the pigeonhole principle, there is some pair $X_i, X_j\subseteq Y_1$ (possibly $i=j$) such that $$e(X_{i}, X_{j})\geq \frac{\alpha^4 n^2 -n}{\binom{\ceiling{k/2}+1}{2}}\geq \frac{\alpha^4}{k^2}n^2\geq \alpha^8n^2.$$
Since $i$ and $j$ have the same parity by design, $t:=\frac{i+j}{2}\leq k\leq \floor{1/\alpha^2}-1$ is an integer.  Now since $\delta(X_{h+1}, X_{h})\geq \alpha^2 n/k$ for all $1\leq h\leq k-1$, we have that $x$ is contained in at least $$\alpha^8n^2\left(\frac{\alpha^2 n}{k}\right)^{i-1}\left(\frac{\alpha^2 n}{k}-1\right)^{j-1}\geq \alpha^{4i+4j}n^{i+j}=(\alpha^4n)^{2t}$$
cycles of length $2t+1$, thus $G$ has the $(2\floor{1/\alpha^2}, \alpha^4)$-vertex-absorbing property.

Now suppose $G$ is $\alpha^4$-near-bipartite.  By Observation \ref{bipartiterobust}, $G$ has  an $(\eta/2, \alpha^2)$-bipartition $\{X, Y\}$ such that $H:=G[X, Y]$ is $(\eta/2, \alpha/2)$-robust.  By Lemma 
\ref{connectingrobust}, $H$ has the $(4/\alpha^2, (\alpha/2)^4)$-connecting property, so for all $x\in X$ and $y\in Y$, there exists $2\leq 2i\leq 4/\alpha^2$ such that there are at least $((\alpha/2)^4 n)^{2i}$ paths with $2i$ internal vertices from $x$ to $y$.
For each such path, there are at least $((\alpha/2)^4 n)^{2i}-2in^{2i-1}\geq \frac{((\alpha/2)^4n)^{2i}}{2}$ paths which are vertex disjoint from the chosen path.  Thus there are at least $((\alpha/2)^4n)^{4i}/4\geq ((\alpha/4)^4n)^{4i}$ cycles containing $x$ and $y$ in which each path from $x$ to $y$ on the cycle has $2i$ internal vertices.
Thus $G$ has the $(4\ell, (\alpha/4)^4)$-pair-absorbing property where $2\ell$ is the largest even integer which is at most $4/\alpha^2$.
\end{proof}

While reading the following technical statement, it is useful to have some idea of how this will be applied in the proof of Lemma \ref{lemma:absorbing}.  For instance, in the non-bipartite case we have by Proposition \ref{manygadgets} that every vertex is contained in many short odd cycles; i.e. a positive proportion of $n^{2i}$ for some small enough $i$.  So in this case, the set $\cT$ will be the vertex set of the graph, the set $\cS_{2i}$ will consist of $(2i)$-tuples of vertices, and in the auxiliary bipartite graph we will put an edge from a vertex to a $(2i)$-tuple if these form a $(2i+1)$-cycle in the original graph.

\begin{proposition}\label{generalabsorbing}
Let $\ell$ be a positive integer, let $0<\sigma_1,\sigma_2, \dots,\sigma_\ell\leq 1$, and let $\sigma:=\min\{\sigma_1,\dots, \sigma_\ell\}$.  For each $1\leq i\leq \ell$, let $\cS_i$ be the collection of all $i$-tuples of distinct elements chosen from $[n]$, let $\cS\subseteq \bigcup_{i\in [\ell]}\cS_i$, and let $\cT$ be any set with $|\cT|\leq |\cS|$.  There exists $n_0$ such that the following holds: If $n\geq n_0$ and $\Gamma$ is an $(\cS,\cT)$-bipartite graph having the property that for all $u\in \cT$ there exists $i\in [\ell]$ such that $$\deg(u, \cS\cap \cS_i)\geq \sigma_i n^i,$$
then there exists a collection of disjoint sets $\cA^*\subseteq \cS$ such that 
$$|\cA^*\cap \cS_i|\leq \frac{\sigma n}{4 \ell^2} ~\text{ for all } i\in [\ell] ~\text{ and }~ \sum_{A\in \cA^*}|A|\leq \sigma n,$$
and for all $u\in \cT$ there exists $i\in [\ell]$ such that 
$$\deg(u, \cA^*\cap \cS_i)\geq \frac{\sigma^2}{32\ell^2} n,$$ 
and $\delta(\cA^*, \cT)\geq 1$.  
Consequently, for all $\cB\subseteq \cT$ with $|\cB|\leq \frac{\sigma^2}{32\ell^2} n$, the subgraph $\Gamma[\cA^*,\cB]$ contains a matching saturating $\cB$. 
\end{proposition}

\begin{proof}
We will show that a randomly chosen subset of $\cS$ will almost surely satisfy all the properties that $\cA^*$ must satisfy.  Then by deleting some elements from the randomly chosen set, we will obtain the actual set $\cA^*$ which has all of the desired properties.

For all $1\leq i\leq \ell$, set $\rho_{i}:=\frac{\sigma}{8\ell^2n^{i-1}}$.
Let $\cA^{(i)}$ be a randomly chosen subset of $\cS_i$ where each each element is chosen independently with probability $\rho_i$ and let $\cA$ be the union of $\cA^{(i)}$ over all $i$.  We note several basic properties of $\cA$ (due to the Chernoff inequality together with the union bound, unless otherwise indicated):  
\begin{itemize}
\item  
With probability at least $1 - \exp \{- n/\log n\}$ we have for all $1\leq i\leq \ell$, 
\begin{equation}\label{Asize}
|\cA^{(i)}|\leq 2\rho_i n^i = \frac{\sigma}{4\ell^2} n ; 
\end{equation}
and thus 
$$\sum_{A\in \cA}|A|= \sum_{i=1}^\ell i\cdot |\cA^{(i)}| \leq \frac{\sigma}{4\ell^2} n\sum_{i=1}^\ell i \leq \frac{\sigma n}{4}\leq \sigma n.$$
\item  
Let $\cA \otimes \cA = \{(S_1, S_2) \in \cA \times \cA: S_1 \cap S_2 \neq \emptyset\}$.  
Then
$$
\bbE \left[ |\cA \otimes \cA| \right] \leq  \sum_{i=1}^\ell \rho_in^i\cdot i\cdot\rho_in^{i-1}\leq \frac{\sigma^2n}{64\ell^4}\sum_{i=1}^\ell i\leq \frac{\sigma^2n}{64\ell^2}.
$$
So by Markov's inequality, 
$$
\Pr \left[ |\cA \otimes \cA| \geq \frac{\sigma^2 n}{32\ell^2}\right] \leq 1/2,   
$$
and thus with probability at least $1/2$, $\cA$ has the property that 
\begin{equation}\label{Aintersect}
|\cA \otimes \cA|<\frac{\sigma^2 n}{64\ell^2}.
\end{equation}

\item For all $u\in \cT$, there exists $i\in [\ell]$ such that $\deg(u, \cS)\geq \sigma_i n^i$.  So with probability at least $1 - \exp \{- n/\log n\}$ we have
\begin{equation}\label{Adegree}
\deg(u, \cA) \geq \frac{1}{2} \rho_i \cdot \sigma_i n^i \geq \frac{\sigma\sigma_i}{16\ell^2} n\geq \frac{\sigma^2}{16\ell^2} n.  
\end{equation}
\end{itemize}

Let $\cA''$ be a subset of $\cS$ for which properties~\eqref{Asize}, \eqref{Aintersect}, and \eqref{Adegree} hold. Now, in every pair of intersecting sets $(S_1, S_2)$ in $\cA''$, delete one of $S_1$ or $S_2$ and let $\mathcal{A}'$ be the resulting set.  By properties \eqref{Aintersect} and \eqref{Adegree}, we have for all $u\in \cT$, there exists $i\in [\ell]$ such that $$\deg(u, \cA') \geq 
\frac{\sigma^2}{16\ell^2} n-\frac{\sigma^2}{32\ell^2} n = \frac{\sigma^2}{32\ell^2}n.$$
Let $\cA^*\subseteq \cA'$ be a maximal subset having the property that $\delta(\cA^*, \cT)\geq 1$ and note that by maximality we still have $\deg(u, \cA^*)\geq\frac{\sigma^2}{32\ell^2}n$.

So for all $\cB\subseteq \cT$ with $|\cB|\leq \frac{\sigma^2}{32\ell^2}n$, we can greedily choose a matching in $\Gamma[\cA^*,\cB]$ which saturates $\cB$.
\end{proof}

\begin{proof}[Proof of Lemma \ref{lemma:absorbing}]
By Proposition \ref{manygadgets}, $G$ either has the $(2\floor{1/\alpha^2},\alpha^4)$-vertex-absorbing property or $G$ is $\alpha^4$-near-bipartite and has the $(4\ell, (\alpha/4)^4)$-pair-absorbing property with $4\ell\leq 8/\alpha^2$.  Suppose first that $G$ has the $(2\floor{1/\alpha^2}, \alpha^4)$-vertex-absorbing property.

Set $2\ell:=2\floor{1/\alpha^2}$, for all $i\in [2\ell]$ set $\sigma_{i}:=(\alpha^4)^{i}$, and set $\sigma:=(\alpha^4)^{2\ell}$. Set $\cT=V(G)$ and $\cS=\{S\in \cS_{2i}: 1\leq i\leq \ell\}$ (recall $\cS_i$ is the set of $i$-tuples of vertices) and let $\Gamma$ be an auxiliary $\cS, \cT$-bipartite graph where $ST$ is an edge if and only if $(v_1,\dots, v_{2i})=S\in \cS$, $x=T\in \cT$, and $xv_1\dots v_{2i}x$ is a cycle of length $2i+1$ in $G$ (i.e. $S$ ``absorbs" $T$).  Since $G$ has the $(2\floor{1/\alpha^2}, \alpha^4)$-vertex-absorbing property, for all $T\in \cT$ we have $\deg(T, \cS\cap \cS_{2i})\geq (\alpha^4n)^{2i}=\sigma_{2i}n^{2i}$ for some $i\in [\ell]$.  So applying Proposition \ref{generalabsorbing} to $\Gamma$ with the parameters $2\ell, \sigma_1,\dots, \sigma_{2\ell},\cS$ defined above, we get a set $\cA^*$ having the stated properties.  Now we will show how to turn $\cA^*$ into the desired path $P^*$.

\begin{figure}[ht]
\centering
\subfloat[]{\label{fig:absorb_a}
\scalebox{.6}{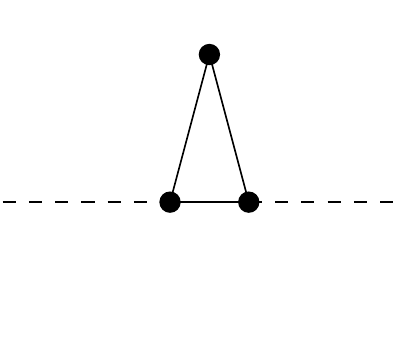}~~~~~~~~
}
\subfloat[]{\label{fig:absorb_b}
\scalebox{.6}{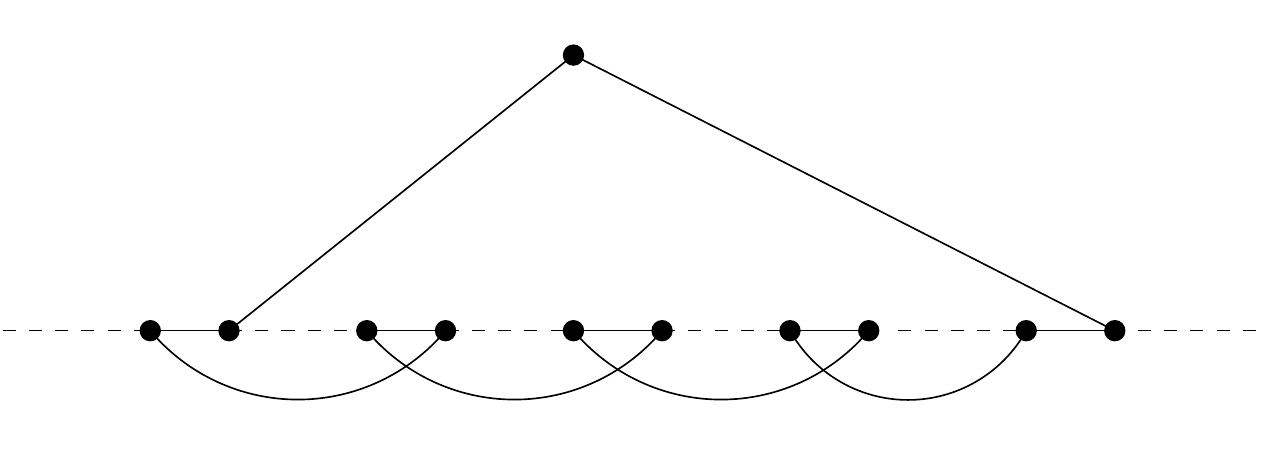}
}
\hspace{.01in}
\subfloat[]{\label{fig:absorb_c}
\scalebox{.6}{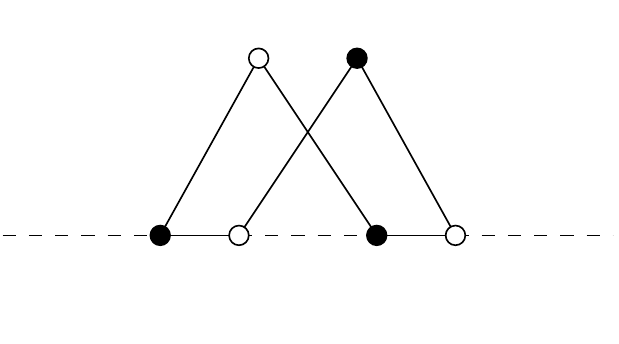}~~~~~
}
\subfloat[]{\label{fig:absorb_d}
\scalebox{.6}{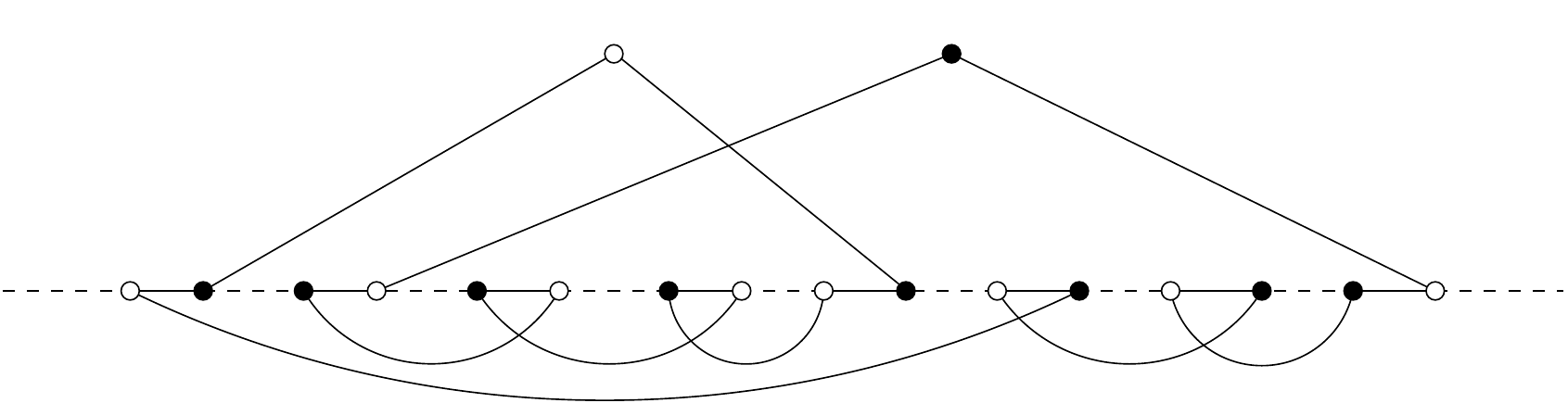}
}
\caption[]{Absorbing vertices and pairs of vertices into cycle segments}\label{fig:absorb}
\end{figure}

For each element $(v_1,\dots, v_{2i})=A\in \cA^*$ we let $P(A)=v_1v_2\dots v_{2i-1}v_{2i}$ be the corresponding path in $G$.  Consider some ordering $A^1,\dots, A^t$ of the elements in $\cA^*$ and suppose we considered all elements up to $A^s=(v_1,\dots, v_{2i})$ where $z^*_{s-1}$ is the last vertex of $A^{s-1}$.  If $i=1$, we set $a^*_s:=v_1$ and $z^*_s:=v_2$ (see Figure \ref{fig:absorb_a}).  If $i\geq 2$ we set $a^*_s:=v_2$, set $z^*_s:=v_{2i}$ and use Lemma \ref{connectingrobust} to build paths of length at most $\ell$ (avoiding all previously used vertices) from $v_{2j-1}$ to $v_{2j+2}$ for all $1\leq j\leq i-2$ and a path from $v_{2i-3}$ to $v_{2i-1}$ (see Figure \ref{fig:absorb_b}).  Finally we build a path from $z^*_{s-1}$ to $a^*_s$ (i.e. the last vertex of $A^{s-1}$ to the first vertex of $A^s$).  So for each $(2i)$-tuple $A\in \cA^*$ we will use at most $2i+i\ell\leq 2i\ell$ vertices and there are at most $\frac{\sigma}{4(2\ell)^2} n$ elements of order $2i$ in $\cA^*$ for each $1\leq i\leq \ell$.  So we have built a path $P^*$ using at most $\sum_{i=1}^\ell 2i\ell\frac{\sigma}{4(2\ell)^2} n\leq \frac{(\ell+1)\sigma}{16}n\leq \rho n$ vertices. 

Now let $W\subseteq V(G)\setminus V(P^*)$ with $|W|\leq \rho^3n\leq \frac{\sigma^2n}{32(2\ell)^2}$.  By Proposition \ref{generalabsorbing}, there is a matching between vertices in $W$ and elements in $\cA^*$ which saturates $W$.  For each $x \in W$, let $A(x) = (v_1, \dots, v_{2i})$ be the element in $\mathcal{A}^*$ matched to $x$.  If $i=1$, then $P^*v_1xv_2P^*$ allows us to insert $x$.  If $i \ge 2$, then $$P^* v_2v_3 \dots v_6v_7 \dots v_{2i-6}v_{2i-5} \dots v_{2i-2}v_{2i-1} \dots v_{2i-3}v_{2i-4} \dots v_9v_8 \dots v_5v_4 \dots v_1xv_{2i} P^*$$ allows us to insert $x$ if $i$ is even, and $$P^* v_2v_3 \dots v_6v_7 \dots v_{2i-4}v_{2i-3} \dots v_{2i-1}v_{2i-2} \dots v_{2i-5}v_{2i-6} \dots v_9v_8 \dots v_5v_4 \dots v_1xv_{2i} P^*$$ allows us to insert $x$ if $i$ is odd.  Since inserting a vertex $x$ rearranges only the internal vertices in the subpath of $P^*$ induced by $A(x)$ to form a new path segment leaving the rest of $P^*$ untouched, we see that $G[V(P^*)\cup W)]$ contains a spanning path having the same endpoints as $P^*$.

Now suppose that $G$ is $\alpha^4$-near-bipartite and has the $(4 \ell, (\alpha/4)^4)$-pair-absorbing property with $4\ell\leq 8/\alpha^2$ which is witnessed by a bipartition $\{ X,Y\}$ such that $H := G[X,Y]$ is $(\eta/2, \alpha/2)$-robust.

For all $i\in [4\ell]$ set $\sigma_{i}:=((\alpha/4)^4)^{i}$, and set $\sigma:=((\alpha/4)^4)^{4\ell}$.  Set $\cT=\{(x,y):x\in X, y\in Y\}$ and $\cS=\{S\in \cS_{4j}: 1\leq j\leq \ell\}$ (recall $\cS_i$ is the set of $i$-tuples of vertices) and let $\Gamma$ be an auxiliary $\cS, \cT$-bipartite graph where $ST$ is an edge if and only if $(a_1, b_1,\dots, a_j,b_j,u_j,v_j, \dots, u_1, v_1)=S\in \cS$, $(x,y)=T\in \cT$, and $$xa_1b_1a_2b_2\dots a_jb_jyv_ju_j\dots v_2u_2v_1u_1x$$ is a cycle of length $4j+2$ in $H$ (i.e. $S$ ``absorbs'' $T$).  Since $H$ has the $(4\ell,(\alpha/4)^4)$-pair-absorbing property, for all $T\in \cT$ we have $\deg(T, \cS\cap \cS_{4j})\geq ((\alpha/4)^4n)^{4j}$ for some $j\in [\ell]$.  So applying Lemma \ref{generalabsorbing} to $\Gamma$ with the parameters $4\ell, \sigma_1,\dots, \sigma_{4\ell},\cS$ defined above, we get a set $\cA^*$ having the stated properties. Now we will show how to turn the set $\cA^*$ into the desired path $P^*$.

Note that each element $A\in \cA^*$ consists of two odd length paths $a_1b_1 \dots a_jb_j$ and $u_1v_1 \dots u_jv_j$.  
Consider some ordering $A^1,\dots, A^t$ of the elements in $\cA^*$ and suppose we considered all elements up to $A^s=(a_1, b_1, \dots, a_j, b_j, u_j, v_j, \dots, u_1, v_1)$ where $z^*_{s-1}$ is the last vertex of $A^{s-1}$.
If $j=1$, we set $a_s^*:=a_1$, set $z^*_{s}:=v_1$, and use Lemma \ref{connectingrobust} (avoiding all previously used vertices) to build a path from $b_1$ to $u_1$ (see Figure \ref{fig:absorb_c})). 
If $j\geq 2$ we set $a^*_s:=b_1$ and $z^*_s:=b_j$ and we use Lemma \ref{connectingrobust} to build paths of length at most $\ell$ from 
$a_1$ to $u_j$, from $a_h$ to $b_{h+1}$ for $2\leq h\leq j-2$, from $a_{j-1}$ to $a_j$, from $v_h$ to $u_{h-1}$ for all $3\leq h\leq j$, from $v_2$ to $v_1$, and from $u_1$ to $b_2$ (see Figure \ref{fig:absorb_d}). Finally, we build a path from $z^*_{s-1}$ to $a^*_s$ (i.e. the last vertex of $A^{s-1}$ to the first vertex of $A^s$). So for each $(4j)$-tuple $A\in \cA^*$ we will use at most $4j+2j\ell\leq 4j\ell$ vertices and there are at most $\frac{\sigma}{4(4\ell)^2} n$ elements of order $4j$ in $\cA^*$ for each $1\leq j\leq \ell$.   So we have built a path $P^*$ using at most $\sum_{j=1}^{\ell} 4j\ell\frac{\sigma}{4(4\ell)^2} n\leq \frac{(\ell+1)\sigma}{32}n\leq \rho n$ vertices. 

Now let $W\subseteq V(G)\setminus V(P^*)$ with $|W\cap X|=|W\cap Y|\leq \rho^3n\leq \frac{\sigma^2n}{32(4\ell)^2}$.  Take an arbitrary partition $\cW_2$ of $W$ into sets of size $2$ such that each member of $\cW_2$ contains one point from $W\cap X$ and one point from $W\cap Y$.  By Proposition \ref{generalabsorbing}, there is a matching between $\cW_2$ and elements in $\cA^*$ which saturates $\cW_2$.  For each $\{x,y\} \in \mathcal{W}_2$, let $A(x,y) = (a_1,b_1, \dots, a_j,b_j,u_j,v_j, \dots, u_1,v_1)$ be the element in $\mathcal{A}^*$ matched to $\{x,y\}$.  If $j=1$, then $P^*a_1xu_1\dots b_1yv_1P^*$ allows us to insert $x$ and $y$.  If $j \ge 2$, then $$P^* b_1a_2 \dots b_{j-1}a_j \dots a_{j-1}b_{j-2} \dots a_3b_2 \dots u_1xa_1 \dots u_jv_{j-1} \dots u_2v_1 \dots v_2u_3 \dots v_jyb_j P^*$$ allows us to insert $x$ and $y$ if $j$ is even, and $$P^* b_1a_2 \dots b_{j-2}a_{j-1} \dots a_jb_{j-1} \dots a_3b_2 \dots u_1xa_1 \dots u_jv_{j-1} \dots u_3v_2 \dots v_1u_2 \dots v_jyb_j P^*$$ allows us to insert $x$ and $y$ if $j$ is odd.  Since inserting a pair $\{x,y\}$ rearranges only the internal vertices in the subpath of $P^*$ induced by $A(x,y)$ to form a new path segment leaving the rest of $P^*$ untouched, we see that $G[V(P^*)\cup W)]$ contains a spanning path having the same endpoints as $P^*$.
\end{proof}

\section{Robust component structure }\label{sec:structure}


\begin{definition}[$(\eta, \alpha)$-nice partition]\label{nicepartition}
Let $0<\alpha\leq \eta/2$ and let $G$ be a $r$-colored graph.  For each $i\in [r]$, let $\mathcal{H}_i$ be a (possibly empty) set of vertex disjoint $(\eta, \alpha)$-robust subgraphs of $G_i$, and let $\mathcal{H}=\bigcup_{i\in [r]}\mathcal{H}_i$.  We say that $\mathcal{H}$ is an $(\eta, \alpha)$-nice partition of $G$ if 
\begin{enumerate}
\item $V(G)=\bigcup_{H\in \mathcal{H}}V(H)$, and
\item for all $H_i\in \mathcal{H}_i$, if $H_i$ is $\alpha^4$-near-bipartite, then there exists an $(\eta, \alpha^2)$-bipartition $\{X, Y\}$ of $H_i$ with $|X|\leq |Y|$ such that $H_i[X,Y]$ is $(\eta/2, \alpha/2)$-robust, and there exists $H_j\in \mathcal{H}_j$ for some $j\neq i$ such that $H_j$ is not $\alpha^4$-near-bipartite and $|V(H_j)\cap Y|\geq \eta^{1/2} n$.
\end{enumerate}
\end{definition}

The second, technical looking condition in the definition above is a direct consequence of the fact that  $\alpha^4$-near-bipartite components can only absorb pairs of vertices from opposite sides of the partition.  It is useful to think of this condition as meaning that some non-bipartite component must be available to ``absorb the imbalance'' from any bipartite component.

In our main lemma of this section, we will show that a $2$-colored graph with $\delta(G)\geq (3/4+\gamma)n$ either directly has the desired monochromatic cycle partition or some robust structure which we will later exploit using regularity and absorbing.

\begin{lemma}[Main structural lemma]\label{robuststructure}
Let $0<\frac{1}{n_0}\ll \alpha\ll \eta\ll \gamma\leq 1/4$ and let $G$ be a $2$-colored graph on $n\geq n_0$ vertices such that $\delta(G)\geq (\frac{3}{4}+\gamma)n$.  
Either $G$ has a partition into a red cycle and a blue cycle or
\begin{enumerate}
\item there exists an $(\eta, \alpha)$-robust subgraph $H_i\subseteq G_i$, such that $|H_i|\geq (1-\eta^{2/3})n$ and $H_i$ is not $\alpha^4$-near-bipartite, or

\item there exist $(\eta, \alpha)$-robust subgraphs $H_i\subseteq G_i$ for $i\in [2]$ such that $\{H_1, H_2\}$ forms an $(\eta, \alpha)$-nice partition of $G$ and 
\begin{enumerate}
\item $|H_1|, |H_2|\geq (3/4+\gamma/2)n$; or
\item $|H_i|\geq (1-\eta^{2/3})n$, $H_i$ is $\alpha^4$-near-bipartite, and $|H_{3-i}|\geq (1/2+\eta^2)n$.
\end{enumerate}
\end{enumerate}
Furthermore, for $i\in[2]$ and all $v\in V(G)\setminus V(H_i)$, $\deg_i(v)<\eta n$.
\end{lemma}

If every vertex had sufficiently large red degree and blue degree, we would have little difficulty proving this lemma.  The first obstacle to overcome is dealing with the vertices which do not have large enough degree in some color. Given a $2$-colored graph $G$, we define $Z_i(G,d)$ to be the set of vertices having degree less than $d$ in color $3-i$, and consequently having degree $\delta(G)-d$ in color $i$.  We refer to the set $Z_i(G,d)$ as the \emph{extreme} vertices of $G_i$.  We now prove two claims which will ultimately be useful in dealing with the extreme vertices.

\begin{claim}\label{Zprune}
Let $0<\frac{1}{n_0}\ll \eta\ll \gamma\leq 1/4$ and let $G$ be a $2$-colored graph on $n\geq n_0$ vertices such that $\delta(G)\geq (\frac{3}{4}+\gamma)n$.  For $i\in [2]$, set $Z_i:=Z_i(G, \eta^{1/3}n)$.  Either 
\begin{enumerate}
\item $\delta_1(G) \geq \eta^{1/3}n$ and $\delta_2(G) \geq \eta^{1/3}n$, or
\item there exists $i\in [2]$ such that $\delta_i(G)\geq \eta^{1/3}n$ and $|Z_i|\geq \eta^{2/3}n$, or
\item $|Z_1|\geq \eta^{2/3}n$ and $|Z_2|\geq \eta^{2/3}n$, or
\item $G':=G-(Z_1\cup Z_2)$ satisfies $|G'|> (1-2\eta^{2/3})n$ and $\delta_1(G') \geq \eta^{1/3}n/2$ and $\delta_2(G') \geq \eta^{1/3}n/2$, or
\item there exists $i\in [2]$ such that $G':=G-Z_{3-i}$ satisfies $|G'|> (1-\eta^{2/3})n$ and $\delta_i(G')\geq \eta^{1/3}n/2$ and $|Z_i'|\geq \eta^{2/3}n$, where $Z_i':=Z_i(G', \eta^{1/3}n)$.
\end{enumerate}
\end{claim}

\begin{proof}
Suppose (i), (ii), and (iii) fail.  If $|Z_1|< \eta^{2/3}n$ and $|Z_2|< \eta^{2/3}n$, then $G':=G-(Z_1\cup Z_2)$ satisfies $|G'|> (1-2\eta^{2/3})n$ and $\delta_1(G') \geq \eta^{1/3}n-2\eta^{2/3}n\geq  \eta^{1/3}n/2$ and $\delta_2(G') \geq \eta^{1/3}n-2\eta^{2/3}n\geq \eta^{1/3}n/2$.

So suppose $|Z_1|\geq \eta^{2/3}n$.  Since (ii) fails, $\delta_1(G)< \eta^{1/3}n$ which implies $|Z_2|>0$.  Since (iii) fails, $0<|Z_2|<\eta^{2/3}n$.  So $G':=G-Z_2$ satisfies $\delta_1(G')\geq \eta^{1/3}n-\eta^{2/3}n\geq  \eta^{1/3}n/2$ and $|Z_1'|\geq \eta^{2/3}n$ (since $Z_1\cap Z_2=\emptyset$).
\end{proof}

\begin{claim}\label{Znotsmall}
Under the same assumptions as in Claim \ref{Zprune}, if $|Z_i|\geq \eta^{2/3}n$, then there exists $H_i\subseteq G_i$ such that $H_i$ is $(\eta, 8\alpha/\eta)$-robust, $Z_i\subseteq V(H_i)$, and $|H_i|\geq (3/4+3\gamma/4)n$.
\end{claim}

\begin{proof}
Without loss of generality, suppose $i=1$.  Note that 
\begin{equation}\label{extremedegree}
\deg_1(v)\geq (3/4+\gamma-\eta^{1/3})n \text{ for all } v\in Z_1.
\end{equation}

First we establish a general bound on the number of common neighbors of color 1 inside a given set $X\subseteq V(G)$. By \eqref{extremedegree}, for all $z,z'\in Z_1$ we have $|N_1(z,z')|\geq (1/2+2\gamma-2\eta^{1/3})n\geq (1/2+\gamma)n$ and thus
\begin{equation}\label{commonoutsideY}
|N_1(z,z')\cap X|\geq |X|-(1/2-\gamma)n.
\end{equation}
If $|Z_1|\geq (3/4+3\gamma/4)n$, then \eqref{commonoutsideY} with $X= Z_1$ implies that for all $z,z'\in Z_1$, 
\begin{align*}
|N_1(z,z')\cap Z_1|\geq (1/4+\gamma)n.
\end{align*}
So $H_1:=G_1[Z_1]$ satisfies $|H_1|\geq (3/4+3\gamma/4)n$ and has the $(1,1/4+\gamma)$-connecting property.  Thus $H_1$ is $(\eta, 8\alpha/\eta)$-robust by Lemma \ref{connectingrobust}.

So suppose $\eta^{2/3}n\leq |Z_1|< (3/4+3\gamma/4)n$.  Let $X=V(G)\setminus Z_1$ and let $X'=\{x\in X: \deg_1(x, Z_1)<\eta^{1/3}|Z_1|\}$. 
Then $\deg_1(v, X)\geq (\frac{3}{4}+\gamma-\eta^{1/3})n-|Z_1|>\gamma n/8$ for all $v\in Z_1$, so by Lemma \ref{markov}.(ii), we have 
\begin{equation}\label{X'size}
|X'|\leq \frac{|Z_1||X|-|Z_1|((\frac{3}{4}+\gamma-\eta^{1/3})n-|Z_1|)}{|Z_1|-\eta^{1/3}|Z_1|}=\frac{|X|+|Z_1|-(\frac{3}{4}+\gamma-\eta^{1/3})n}{1-\eta^{1/3}}<(\frac{1}{4}-\frac{\gamma}{4})n.
\end{equation}
Set $U:=Z_1\cup (X\setminus X')$ and note that $|U|=n-|X'|> (3/4+3\gamma/4)n$.  Set $H_1:=G_1[U]$.  Note that $\deg_1(u, Z_1)\geq \eta^{1/3}|Z_1|\geq \eta n$ for all $u\in X\setminus X'$, and also by \eqref{commonoutsideY} and \eqref{X'size} we have $|N_1(z,z')\cap U|\geq (3/4+3\gamma/4)n-(1/2-\gamma)n\geq (1/4+\gamma)n$ for all $z,z'\in Z_1$.  Thus for every pair of vertices in $U$, we can either find $(1/4+\gamma)n$ paths of length $2$, $\eta (1/4+\gamma)n^2$ paths of length $3$, or $\eta^2 (1/4+\gamma)n^3$ paths of length $4$.  Combined with the fact that $\delta(H_1)\geq \eta n$, we may apply Lemma \ref{connectingrobust}(ii) to see that $H_1$ is $(\eta, 8\alpha/\eta)$-robust.  
\end{proof}

We now prove a preliminary result which applies to graphs $G$ having the property that for all $i\in [2]$, either $\delta_i(G)$ is sufficiently large or the number of extreme vertices in $G_{3-i}$ is sufficiently large (i.e. Claim \ref{Zprune}(i),(ii), or (iii) holds).  Treating this case separately will allow for Lemma \ref{robuststructure} to have a cleaner proof.

\begin{proposition}\label{prerobustlemma}
Let $0<\frac{1}{n_0}\ll \alpha\ll \eta\ll \gamma\leq 1/4$ and let $G$ be a $2$-colored graph on $n\geq n_0$ vertices such that $\delta(G)\geq (\frac{3}{4}+\gamma)n$. If 
\begin{enumerate}[label=(\alph*)]
\item $\delta_1(G)\geq \eta^{1/3}n$ and $\delta_2(G) \geq \eta^{1/3}n$, or
\item there exists $i\in [2]$ such that $\delta_i(G)\geq \eta^{1/3}n$ and $|Z_i|\geq \eta^{2/3}n$, or
\item $|Z_1|\geq \eta^{2/3}n$ and $|Z_2|\geq \eta^{2/3}n$,
\end{enumerate}
then there exist $(\eta, \alpha)$-robust subgraphs $H_i\subseteq G_i$ for all $i \in [2]$ such that
\begin{enumerate}
\item $|H_1|,|H_2|\geq (3/4+3\gamma/4)n$ and $V(G)=V(H_1)\cup V(H_2)$, or
\item $|H_i|=n$ for some $i\in [2]$.
\end{enumerate}
\end{proposition}

\begin{proof}
Suppose that for all $i\in [2]$ we have $\delta_i(G)\geq \eta^{1/3}n$ or $|Z_{3-i}| \geq \eta^{2/3}n$ (this is just a concise way of stating the hypothesis). 

Suppose first that the largest monochromatic $(\eta, 3\alpha/\eta)$-robust subgraph of $G$ has fewer than $(1/2-\gamma)n$ vertices.  In this case, we must have $|Z_1|, |Z_2|< \eta^{2/3}n$ because of Claim \ref{Znotsmall}, which means $\delta_1(G)$, $\delta_2(G)\geq \eta^{1/3}n$.  Now apply Observation \ref{robustpartition} to each of $G_1$ and $G_2$ to get a $(\eta^{2/3}/4, \eta/80)$-robust partition of $G_i$ for $i\in[2]$.  Choose $H_1'\subseteq G_1$ and $H_2'\subseteq G_2$ to be the pair of parts in the partition having maximum intersection.  Since (as given by Observation \ref{robustpartition}) there are at most $2/\eta^{1/3}$ parts in the partition of each $G_i$ and each part has size at least $\eta^{1/3}n/2$, we have 
\begin{equation}\label{bigintersection}
|V(H_1')\cap V(H_2')|\geq \eta^{2/3} n/4. 
\end{equation}
Let $H_1$ and $H_2$ be $\eta$-maximal extensions of $H_1'$ and $H_2'$ respectively.  We note that since $|H_i'|\geq \eta^{1/3}n/2$, we may apply Observation \ref{augmentrobust} with $\alpha=\eta/80$, with $\eta'=\eta$ and with $\tau\geq \eta^{1/3}/2$ to get that $H_i$ is $(\eta, \eta^{7/3}/160)$-robust for $i\in [2]$.

If $|V(H_1)\cup V(H_2)|<3n/4$, let $L:=\{v\in V(H_1)\cap V(H_2): \deg(v, V(G)\setminus (V(H_1)\cup V(H_2))\geq \gamma n\}$.  Since $H_1$ and $H_2$ are $\eta$-maximal, we have $e(V(H_1)\cap V(H_2), V(G)\setminus (V(H_1)\cup V(H_2))<\eta n^2$ and thus $|L|<\frac{\eta}{\gamma}n<\eta^{2/3} n/4$.  Thus by \eqref{bigintersection}, $(V(H_1)\cap V(H_2))\setminus L$ is non-empty.  Now for all $v\in (V(H_1)\cap V(H_2))\setminus L$ we have $\deg(v)< |V(H_1)\cup V(H_2)|+ \gamma n<(3/4+\gamma)n$, a contradiction.

So $|V(H_1)\cup V(H_2)|\geq 3n/4$.  By assumption, we have $|H_1|, |H_2|<(1/2-\gamma)n$ (since $H_i$ is $(\eta, \eta^{7/3}/160)$-robust for $i\in [2]$).  For a vertex $v\in V(H_1)\setminus V(H_2)$, we have $\deg(v)\leq |H_1|+\eta n+|V(G)\setminus (V(H_1)\cup V(H_2))|\leq (1/2-\gamma)n+\eta n+n/4<3n/4$, a contradiction.

Next we show that if $G$ contains a monochromatic $(\eta, 3\alpha/\eta)$-robust subgraph on at least $(1-\eta^{1/3}+\eta)n$ vertices, then we satisfy conclusion (i) or (ii).  So without loss of generality, suppose $H_2'\subseteq G_2$ is such a subgraph.  Let $H_2$ be an $\eta$-maximal extension of $H_2'$. 
If $|H_2|=n$, we satisfy conclusion (ii), so suppose not.  In this case we have $\deg_2(v)<|V(G)\setminus V(H_2)|+\eta n<\eta^{1/3}n$ for all $v\in V(G)\setminus V(H_2)$ and thus $V(G)\setminus V(H_2)\subseteq Z_1$.  Since $\delta_2(G)<\eta^{1/3}n$, we must have $|Z_1|>\eta^{2/3}n$ (by the original assumption) and thus by Claim \ref{Znotsmall}, there exists $H_1$ which together with $H_2$ satisfies conclusion (i). 

So we are in the case where there exists a monochromatic $(\eta, 3\alpha/\eta)$-robust subgraph, say  $H_2'\subseteq G_2$ with at least $(1/2-\gamma)n$ vertices, and no monochromatic $(\eta, 3\alpha/\eta)$-robust subgraph of $G$ has more than $(1-\eta^{1/3}+\eta)n$ vertices.

Let $H_2$ be an $\eta$-maximal extension of $H_2'$. By Observation \ref{augmentrobust}, $H_2$ is $(\eta, \alpha)$-robust (in the application of Observation \ref{augmentrobust}, we have $\tau>1/3$ since $|H_2'|\geq (1/2-\gamma)n$).  Set $b:=|H_2|$ and note that $(1/2-\gamma)n\leq b\leq (1-\eta^{1/3}+\eta)n$.

We first note that there can only be a small number of vertices in $H_2$ which have at least $\eta^{1/2}n$ neighbors of color $2$ outside of $H_2$. Formally, set $Y_1:=V(G)\setminus V(H_2)$ and let $L_2=\{v\in V(H_2):\deg_2(v, Y_1)\geq \eta^{1/2}n\}$.  Since $e_2(Y_1, V(H_2))<\eta n^2$ (because $H_2$ is $\eta$-maximal), we have $|L_2|\leq \eta^{1/2}n$.

For all $v\in Y_1$, 
\begin{equation}\label{singleinH2}
\deg_1(v, V(H_2))\geq (3/4+\gamma)n-\eta n-(n-b)=b-(1/4-\gamma+\eta)n.
\end{equation}
So for all $u,v\in Y_1$, since $b\geq (1/2-\gamma)n$, we have
\begin{equation}\label{commoninH_2}
|N_1(u,v)\cap V(H_2)|\geq 2(b-(1/4-\gamma+\eta)n)-b=b-(1/2-2\gamma+2\eta)n\geq \gamma n/2.
\end{equation}
Also, for all $v\in V(H_2)\setminus L_2$, 
\begin{equation}\label{smallb}
\deg_1(v, Y_1)\geq (3/4+\gamma)n-\eta^{1/2}n-b=(3/4+\gamma-\eta^{1/2})n-b
\end{equation}

So if $b\leq (3/4+3\gamma/4)n$, then \eqref{smallb}, \eqref{commoninH_2}, and the bound on $|L_2|$ imply that for $H_1':=G_1[V(G)\setminus L_2]$, we have $\delta(H_1')\geq \gamma n/8$, $|H_1'|\geq (1-\eta^{1/2})n\geq (1-\eta^{1/3}+\eta)n$, and for every pair of vertices $u,v\in V(H_1')$, there is an $1\leq i\leq 3$ such that there are at least $(\gamma n/8)^i$ paths of length $i+1$ from $u$ to $v$.  Applying Lemma \ref{connectingrobust}(ii), we see that $H_1'$ is $(\gamma/8, (\gamma/8)^3)$-robust.  However, we are in the case where there is no such robust monochromatic subgraph of this size.  

So suppose $(3/4+3\gamma/4)n\leq b<(1-\eta^{1/3}+\eta)n$.  Let $X_2=\{v\in V(H_2): \deg_1(v, Y_1)<\eta n\}$.  By \eqref{singleinH2} and Lemma \ref{markov}.(ii), we have 
\begin{align*}
|X_2|\leq \frac{(n-b)b-(b-(\frac{1}{4}-\gamma+\eta)n)(n-b)}{n-b-\eta n}
=\frac{(n-b)(\frac{1}{4}-\gamma+\eta)n}{n-b-\eta n} 
\leq (1/4-\gamma/4)n.
\end{align*}

Then $H_1:=G_1[Y_1\cup (V(H_2)\setminus X_2)]=G_1[V(G)\setminus X_2]$ satisfies $|V(H_1)|\geq (3/4+3\gamma/4)n$ and is $(\eta, 3\alpha/\eta)$-robust in color $1$ in which case we satisfy conclusion (i). Indeed, by \eqref{commoninH_2} and the upper bound on $|X_2|$ we have for all $u,v\in Y_1$, 
$$|N_2(u,v)\cap (V(H_2)\setminus X_2)|\geq b-(1/2-2\gamma+2\eta)n-(1/4-\gamma/4)n \geq \gamma n$$ 
and $\deg_1(v, Y_1)\geq \eta n$ for all $v\in V(H_2)\setminus X_2$. Thus for every pair of vertices $u,v\in V(H_1)$ there exists $1\leq i\leq 3$ such that there are at least $(\eta n)^i$ paths of length $i+1$ from $u$ to $v$; thus by Lemma \ref{connectingrobust}(ii), $H_1$ is $(\eta, 3\alpha/\eta)$-robust.

\end{proof}

Finally, we prove the main result of this section.

\begin{proof}[Proof of Lemma \ref{robuststructure}]
Set $\alpha_0=4\alpha/\eta$ and $\gamma_0=\gamma-8\eta^{2/3}$ and start by applying Claim \ref{Zprune} with $8\eta$ and $\gamma$.  If Claim \ref{Zprune} (i), (ii), or (iii) hold, then set $G':=G$.  If Claim \ref{Zprune} (iv) holds, set $G':=G-(Z_1\cup Z_2)$.  If Claim \ref{Zprune} (v) holds, then set $G':=G-Z_{3-i}$.  Set $n':=|G'|$.  Note that $G'$ satisfies the hypotheses of Proposition \ref{prerobustlemma} (with $n'$, $\alpha_0$, $\eta$, $\gamma_0$) so we  obtain $(\eta, \alpha_0)$-robust subgraphs $H_1'$, $H_2'$ of $G'$ satisfying the conclusion of Proposition \ref{prerobustlemma}.

First suppose Proposition \ref{prerobustlemma}.(i) holds; that is, $|H_1'|, |H_2'|\geq (3/4+3\gamma_0/4)n\geq (3/4+\gamma/2)n$ and $V(G')=V(H_1')\cup V(H_2')$.  Then every vertex $v\in Z_1\cup Z_2$ satisfies $\deg_i(v,H_i')\geq 3n/8$ for some $i\in [2]$.  We add these vertices to the appropriate components by taking $\eta$-maximal extensions and thus by Observation \ref{augmentrobust}, $H_1$ and $H_2$ are $(\eta, \alpha)$-robust and satisfy conclusion (ii.a).  If both $H_1$ and $H_2$ are not $\alpha^4$-near-bipartite, then $H_1, H_2$ forms the desired $(\eta, \alpha)$-nice partition.  We delay the proof when, say $H_1$ is $\alpha^4$-near-bipartite until the end (see Case 1 below).

Now suppose Proposition \ref{prerobustlemma}.(ii) holds; that is, without loss of generality $|H_1'|=n'$. If $G_2'$ contains an $(\eta, \alpha_0)$-robust subgraph with $|H_2'|\geq (3/4+3\gamma_0/4)n$, we would be in the previous case; so suppose not.  Note that for all $v\in Z_1$, we have $\deg_1(v, H_1')\geq 3n/4$, so adding these vertices to $H_1'$ by taking an $\eta$-maximal extension of $H_1'$ and applying Observation \ref{augmentrobust} gives a $(\eta, \alpha)$-robust component $H_1$.  Note that any vertices in $G-H_1$ must be in $Z_2$ and if $|Z_2|\geq \eta^{2/3}n$, then by Claim \ref{Znotsmall}, we would have a $(\eta, \alpha_0)$-robust subgraph with $|H_2'|\geq (3/4+3\gamma_0/4)n$, which we don't have in this case.  So $|Z_2|< \eta^{2/3}n$ and $V(G)\setminus V(H_1)\subseteq Z_2$.  If $H_1$ is not $\alpha^4$-near-bipartite, then we satisfy conclusion (i).

We have shown that either conclusion (i) or (ii.a) holds, but if some $H_i$ is $\alpha^4$-near-bipartite, then additional properties must hold in order to get an $(\eta, \alpha)$-nice partition.  From the cases above, assume that either $|H_1|, |H_2|\geq (3/4+\gamma/2)n$ and without loss of generality $H_1$ is $\alpha^4$-near-bipartite or $|H_1|\geq (1-\eta^{2/3})n$ and $H_1$ is $\alpha^4$-near-bipartite.  Since $H_1$ is $(\eta, \alpha)$-robust and $\alpha^4$-near-bipartite, by Observation \ref{bipartiterobust} there exists an $(\eta/2, \alpha^2)$-bipartition $\{S, T\}$ of $H_1$ with $|S|\leq |T|$ such that $H_1[S,T]$ is $(\eta/2, \alpha/2)$-robust.  Let $T'=\{v\in T:\deg_1(v, T)\geq \alpha n\}$, $S'=\{v\in S:\deg_1(v, S)\geq \alpha n\}$, $U'=\{v\in V(H_1): \deg_1(v, V(G)\setminus V(H_1))\geq \gamma n/4\}$.  Since $e(S), e(T)\leq \alpha^2 n^2$, we have $|S'|, |T'|\leq \alpha n$.  Since every vertex in $V(G)\setminus V(H_1)$ has fewer than $\eta n$ neighbors of color 1 in $H_1$, we have $e_1(V(H_1), V(G)\setminus V(H_1))\leq \eta n^2$ and thus $|U'|\leq \frac{4\eta}{\gamma}n$.  We will show that there exists a set $X$ which contains all of the vertices of $V(G)\setminus V(H_1)$ and most of the vertices of $T$ and which will be contained in our robust component $H_2$ which is not near-bipartite (this will ensure that $H_1, H_2$ will form the $(\eta, \alpha)$-nice partition). 

Precisely, let $X=V(G)\setminus (S\cup T'\cup U')$.  Note that by the bounds on $T'$ and $U'$, we have $|X\cap T|\geq |T|-\frac{4\eta}{\gamma}n-\alpha n\geq (1-\gamma)|T|$; once we show that $X\subseteq V(H_2)$ and $H_2$ is not $\alpha^4$-near-bipartite, this shows that condition (ii) in Definition \ref{nicepartition} is satisfied. In both of the following cases we will need some observations about the degree of vertices in $X$.  Using the fact that either $x\in X\setminus T$ and $H_1$ is $\eta$-maximal, or $x\in X\cap T$ and thus not in $U'$ or $T'$, we have for all $v\in X$, 
\begin{equation}\label{XtoX}
\deg_2(v, X)\geq (3/4+\gamma)n-\gamma n/4-\alpha n-|S\cup T'\cup U'|\geq (3/4+\gamma/2)n-|S|
\end{equation}
 and
\begin{equation}\label{XtoT}
\deg_2(v, X\cap T)\geq (3/4+\gamma)n-\gamma n/4-\alpha n-(n-|H_1|)-|S\cup T'\cup U'|\geq |X\cap T|-(1/4-\gamma/2) n.
\end{equation}

\noindent
\textbf{Case 1} ($|H_1|, |H_2|\geq (3/4+\gamma/2)n$).  First note that since $H_1$ is $\eta$-maximal, for all $v\in X\setminus T$ we have $\deg_2(v)\geq (3/4+\gamma)n-(n-|H_1|)-\eta n\geq (1/2+\gamma)n$, and for all $v\in X\cap T$, by \eqref{XtoX} we have $\deg_2(v)\geq (3/4+\gamma/2)n-|S|\geq (1/4+\gamma/2)n$.  Thus $\deg_2(v, H_2)\geq \gamma n$ for all $v\in X$, and since $H_2$ is $\eta$-maximal we have $X\subseteq V(H_2)$.

So the only thing left to check is that $H_2$ is not $\alpha^4$-near-bipartite.  Note that $|T|\geq |H_1|/2\geq (3/8+\gamma/4)n$.  By \eqref{XtoT}, for all $v\in X\cap T$ we have $\deg_2(v, X\cap T)\geq |X\cap T|-(1/4-\gamma/2) n\geq n/8$, and by \eqref{XtoX}, for all $v,v'\in X\cap T$ we have $|N_2(v,v')\cap X|\geq 2((3/4+\gamma/2)n-|S|)-(n-|S|)=(1/2+\gamma)n-|S|\geq \gamma n$, thus there are at least $\frac{1}{6}|X\cap T|\cdot \frac{n}{8} \cdot \gamma n>\alpha n^3$ triangles in $H_2$ which implies that $H_2$ is not $\alpha^4$-near-bipartite.

\noindent
\textbf{Case 2} ($|H_1|\geq (1-\eta^{2/3})n$).  First note that because of the order of $H_1$ and the definition of $U'$, we have $U'=\emptyset$.  Let $v,v'\in X$; by \eqref{XtoT} and \eqref{XtoX}, we have $|N_2(v, v')\cap X|\geq 2((3/4+\gamma/2)n-|S|)-(n-|S|)\geq (1/2+\gamma)n-|S|\geq \gamma n$.  Thus $G_2[X]$ has the $(1,\gamma)$-connecting property and is $(\gamma, \gamma^2)$-robust by Lemma \ref{connectingrobust}.  Note that since $G_2[X]$ has the $(1,\gamma)$-connecting property, every edge of $G_2[X]$ is contained in at least $\gamma n$ triangles and thus $G_2[X]$ contains at least $\gamma^2n^3/3>\alpha n^3$ triangles.  Taking an $\eta$-maximal extension of $G_2[X]$ gives an $(\eta, \alpha)$-robust subgraph $H_2\subseteq G_2$ which is not $\alpha^4$-near-bipartite.  At this point, $H_1, H_2$ satisfy the necessary conditions to form an $(\eta, \alpha)$-nice partition.  The only thing left to check is that $|H_2|\geq (1/2+\eta^2)n$.  

Suppose $|H_2|< (1/2+\eta^2)n$ and note that we have $(1/2+\eta^2)n>|H_2|\geq n-|S|-|T'|$.
In particular, this implies 
\begin{equation}\label{ST}
(1/2-2\eta^2)n\leq |S|\leq |T|\leq (1/2+2\eta^2)n ~\text{ and }~ |H_1|\geq (1-4\eta^2)n.
\end{equation} 
The above calculations show that $H_1[S,T]$ is a nearly balanced bipartite graph.  We will now show that  most vertices have degree greater than $|S|/2$ or $|T|/2$ respectively.  This will allow us to construct the desired monochromatic cycle partition directly.

\begin{figure}[h]
\centering
\includegraphics{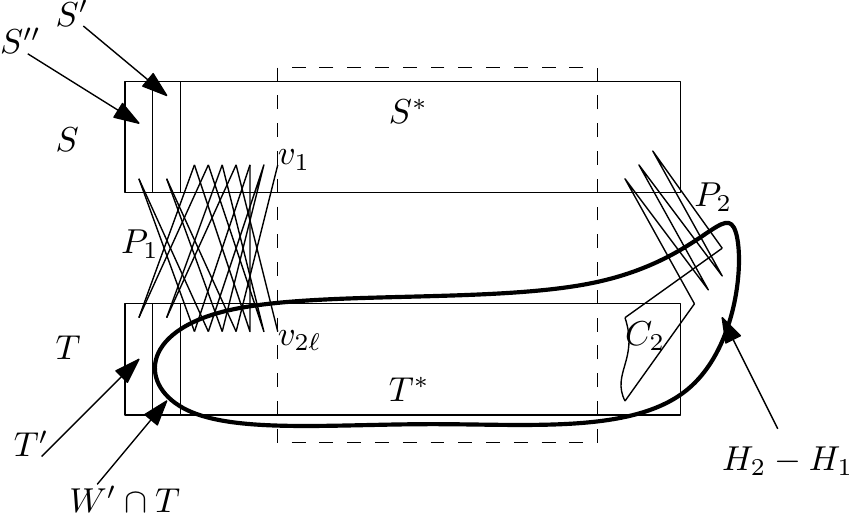}
\caption{Constructing the monochromatic cycle partition directly.}
\label{H1fig}
\end{figure}

Let $W'=\{v\in V(H_2): \deg_2(v, V(H_1)\setminus V(H_2))\geq \gamma n/4\}$ and note that since $H_2$ is $\eta$-maximal, there are fewer than $\eta n^2$ edges of color $2$ between $V(H_2)$ and $V(H_1)\setminus V(H_2)$, and therefore $|W'|\leq \frac{4\eta}{\gamma} n$.  Let $S''=S\cap V(H_2)$ and note that because of the fact that $|V(H_2)\cap T|\geq |T|-|T'|\geq (1/2-3\eta^2)n$ and $|H_2|<(1/2+\eta^2)n$, we have $|S''|<4\eta^2n$. 

\begin{claim}\label{H*}
Let $S^*\subseteq S\setminus (S'\cup S'')$ and $T^*\subseteq T\setminus (T'\cup W')$ such that $|S^*|=|T^*|\geq |S|-\gamma/2 n$.  Then for all $x\in S^*$, $y\in T^*$ there is a Hamiltonian path in $H_1[S^*, T^*]$ having $x$ and $y$ as endpoints (i.e. $H_1[S^*, T^*]$ is Hamiltonian biconnected).
\end{claim}

\begin{proof}
Indeed, for all $v\in S^*$, using \eqref{ST} we have $\deg_2(v, T)\leq \deg_2(v, H_2)+|T'|\leq (\eta+\alpha)n$ and thus 
$$\deg_1(v, T^*)\geq (3/4+\gamma)n-(n-|T^*|) - \deg_2(v,T) \geq (1/4+\gamma/4)n>|T^*|/2+1.$$
For all $v\in T^*$, we have $\deg_2(v, S^*)\leq \gamma n/4$ and thus 
$$\deg_1(v, S^*)\geq (3/4+\gamma)n-(n-|S^*|) - \deg_2(v,S) \geq (1/4+\gamma/8)n>|S^*|/2+1.$$
So let $x\in S^*$ and $y\in T^*$ and apply Lemma \ref{ChvatalDegree} to $H_1[S^*\setminus \{x\}, T^*\setminus \{y\}]$ to get a Hamiltonian cycle $C'$.  Now using the degree condition there exist consecutive vertices $v_i, v_{i+1}$ on $C'$ such that $x$ is adjacent to $v_{i+1}$ and $y$ is adjacent to $v_i$, showing that $H_1[S^*, T^*]$ is Hamiltonian biconnected.
\end{proof}

Now we build a path $P_1$ and a cycle $C_2$ so that $C_2$ contains all of the vertices in $H_2-H_1$ and its deletion leaves $H_1[S,T]$ balanced and so that $P_1$ contains all of the vertices in $S'\cup S''\cup T'\cup W'$ and has its endpoints outside of $S'\cup S''\cup T'\cup W'$.  We have that $|S''|\leq 4\eta^2 n$, that $|W'|\leq \frac{4\eta}{\gamma} n$ and that $|S'|,|T'|\leq \alpha n$.  Since the the minimum degree in $H_1[S, T]$ is at least $\eta n/2$ and all of the vertices not in $S'\cup S''\cup T'\cup W'$ satisfy  degree conditions from Claim \ref{H*}, we can greedily find a path $P_1=v_1\dots v_{2\ell}$ in $H_1[S, T]$ such that (i) $S''\cup S'\cup W'\cup T'\subseteq V(P_1)$, (ii) $v_1\in S\setminus (S'\cup S'')$ and $v_{2\ell}\in T\setminus (T'\cup W')$, and (iii) $|P_1|\leq 3|S'\cup S''\cup T'\cup W'|$ (see Figure \ref{H1fig}).
For all $v\in V(G)\setminus V(H_1)$, we have by \eqref{ST} that $\deg_2(v, S)\geq (3/4+\gamma)n-\eta n-(n-|S|)\geq (1/4+\gamma/2)n$ and thus $|N_2(v,v')\cap S|\geq \gamma n$ for all $v,v'\in V(G)\setminus V(H_1)$.  So we greedily find a blue path $P_2$ having $|V(G)\setminus V(H_1)|$ vertices in  $V(G)\setminus V(H_1)$ and $|V(G)\setminus V(H_1)|-1$ vertices in $S\setminus V(P_1)$ (with both endpoints in $V(G)\setminus V(H_1)$); note that by \eqref{ST} we have $|P_2|\leq 2\cdot 4\eta^2n-1\leq 8\eta^2n$.  Now, use the fact that $G_2[X]\subseteq H_2$ has the $(1,\gamma)$-connecting property to greedily extend $P_2$ into a cycle $C_2$ using vertices from $T\setminus V(P_1)$ so that $S^*:=S\setminus (\{v_2,\dots, v_{2\ell-1}\}\cup V(C_2))$ and $T^*:=T\setminus (\{v_2,\dots, v_{2\ell-1}\}\cup V(C_2))$ satisfy $|S^*|=|T^*|$.  Note that by \eqref{ST} and the bound on $|P_2|$, we have $|T|-(|S \setminus V(P_2)|)\leq 8\eta^2 n$ and thus $|C_2|\leq 8\eta^2n+|P_2|\leq 16\eta^2 n$.
All together, we have 
$$|P_1|+|C_2|\leq 3|S'\cup S''\cup T'\cup W'|+16\eta^2 n \leq (12\eta^2+12\eta/\gamma+6\alpha+16\eta^2) n\leq \gamma n/2.$$
So, applying Claim \ref{H*} we are able to extend $P_1$ into a cycle $C_1$ completing the desired monochromatic cycle partition.  
\end{proof}

\section{From connected matchings to cycles}\label{sec:proof}

\subsection{Connected matchings}

We will need the following preliminary lemma.

\begin{lemma}\label{partitematching}
Let $n$ be even and $k\geq 2$, and let $G$ be a $k$-partite graph on $n$ vertices with the vertex set partitioned as $\{X_1,X_2,\dots, X_k\}$.  Suppose that $|X_i|\leq \frac{n}{2}$ for all $i\in [k]$.  If $\deg(x)> \frac{3}{4}n-|X_i|$ for all $i\in [k]$ and for all $x\in X_i$, 
then $G$ is connected and contains a perfect matching.  
\end{lemma}

\begin{proof}

Consider an edge maximal counterexample $G$.  First note that $G$ cannot be a complete $k$-partite graph (that is, a complete $k$-partite graph with all part sizes at most $n/2$ must have a perfect matching).  Indeed, if $G$ is a complete $k$-partite graph, let $\{Y_1,\dots, Y_k\}=\{X_1,\dots, X_k\}$ such that $|Y_1|\geq\dots\geq |Y_k|$.  Consider a balanced bipartition $\{A_1,A_2\}$ of $V(G)$ which satisfies $Y_1\subseteq A_1$ and $Y_2\subseteq A_2$.  Suppose there exists a vertex $v\in A_j$ such that $\deg(v,A_{3-j})\leq |A_{3-j}|/2=n/4$.  This implies that $v\in Y_i$ for some $i\geq 3$ and $|Y_i|>|Y_i\cap A_{3-j}|\geq |A_{3-j}|/2$ which implies $|Y_{3-j}|\leq |A_{3-j}|/2$; however, $|Y_1|\geq |Y_2|\geq |Y_i|$, a contradiction.  Thus $\delta(G[A_1,A_2])>n/4$ which implies $G[A_1, A_2]$ is connected and has a perfect matching (by applying Lemma \ref{ChvatalDegree}.(ii) for instance).

Since our edge maximal counterexample $G$ cannot be complete, adding any edge between parts gives a perfect matching.  So without loss of generality, let $v_1\in X_1$ and $v_2 \in X_2$ such that $e:=v_1v_2\not\in E(G)$ and let $M$ be a perfect matching in $G+e$.  For all $1\leq i<j\leq k$, let $M_{i,j}$ be the set of edges in $M$ between $X_i$ and $X_j$, let $m_{i,j}:=|M_{i,j}|$, and let $x_i:=|X_i|$.  We first deduce some facts about the sizes of the parts based on the matching $M$.

If we were to delete the vertices of all matching edges inside $X_1\cup X_2$, the remaining vertices in $X_1 \cup X_2$ must be matched by edges in $M \setminus M_{1,2}$ to vertices in $X_3\cup\dots\cup X_k$.  Hence $$x_1 + x_2 - 2m_{1,2} \le |M| - m_{1,2},$$ which rearranging yields
\begin{equation}\label{rearrange}
-m_{1,2} \le |M| - x_1 - x_2.
\end{equation} 

Now remove $e$ and consider $v_1$ and $v_2$. If there was an edge $ab\in M-v_1v_2$ such that $a\in N(v_1)$ and $b\in N(v_2)$, then we would have a perfect matching in $G$; so suppose not.  This implies that $\deg_{G+e}(v_1, e')+\deg_{G+e}(v_2, e')\leq 2$ for all $e'\in M$, and that $\deg_{G+e}(v_1, e')+\deg_{G+e}(v_2,e')\leq 1$ for all $e'\in M_{1,2}-e$.  Using this fact together with \eqref{rearrange} and the minimum degree condition we get,
\begin{align*}
\frac{3}{2}n-x_1-x_2< \deg_G(v_1)+\deg_G(v_2)\leq (m_{1,2}-1) + 2(|M| - m_{1,2})
&\le 3|M| - x_1 - x_2 - 1\\
&=\frac{3}{2}n-x_1-x_2 - 1,
\end{align*}
a contradiction.

Finally to see that $G$ is connected, let $u,v\in V(G)$.  If $u,v\in X_i$ for some $i \in [k]$, then $|N(u)\cap N(v)|>2(3n/4-x_i)-(n-x_i)=n/2-x_i\geq 0$.  So suppose $u\in X_i$ and $v\in X_j$ with $i\neq j$.  If either $v$ has a neighbor $v'$ in $X_i$ or $u$ has a neighbor $u'\in X_j$, then $u$ and $v'$ have a common neighbor or $u'$ and $v$ have a common neighbor.  Otherwise, neither $u$ nor $v$ have neighbors in $X_i \cup X_j$ and thus $|N(u)\cap N(v)|>(3n/4-x_i)+(3n/4-x_j)-(n-x_i-x_j)= n/2$.
\end{proof}

The following lemma will be applied in the reduced graph, and for our purposes it is convenient for us to allow an edge to be colored with both red and blue.  So recall that  $\{E_1, E_2\}$ is a $2$-multicoloring of $G$ if $E_1\cup E_2=E(G)$ (i.e. we allow for $E_1\cap E_2\neq \emptyset$).

\begin{lemma}\label{pureconnectedmatching}
Let $G$ be a graph on $n$ vertices with $n$ even such that $\delta(G)\geq 3n/4$ and let $E_1\cup E_2$ be a $2$-multicoloring of $G$.  For all components $H_1\subseteq G_1$ and $H_2\subseteq G_2$, if 
\begin{enumerate}
\item $|H_1|, |H_2|\geq 3n/4$ and $V(G)=V(H_1)\cup V(H_2)$, or

\item $|H_i|=n$ and $H_{3-i}$ is the largest component of $G_{3-i}$,
\end{enumerate}
then $G$ contains a perfect matching $M\subseteq E(H_1)\cup E(H_2)$.  Furthermore, if $|H_i|=n$ and $|H_{3-i}|\leq n/2$, then $M\subseteq E(H_i)$.
\end{lemma}

\begin{proof}

Without loss of generality, suppose $|H_1|\geq |H_2|$.  Let $G'$ be the $2$-multicolored graph obtained from $G$ by doing the following to each edge $e$ which has both endpoints in $V(H_i)\setminus V(H_{3-i})$ for some $i\in [2]$: if $e$ is colored with both $1$ and $2$, remove color $3-i$; if $e$ is only colored with $3-i$, then delete $e$.

\noindent
\textbf{Case 1} ($|H_1|,|H_2|\geq 3n/4$ and $V(G)=V(H_1)\cup V(H_2)$).  
Since $|V(H_i)\setminus V(H_{3-i})|\leq n/4$ for $i\in [2]$, we have $\delta(G')\geq n/2$ and thus Lemma \ref{ChvatalDegree} (or more simply, Dirac's theorem) implies that $G'$ has a Hamiltonian cycle.  The edges of the Hamiltonian cycle lie entirely inside $E(H_1)\cup E(H_2)$ giving us the desired matching.

\noindent
\textbf{Case 2} ($|H_1|=n$)  Suppose first that $G_2$ contains a component $H_2$ such that $|H_2|> n/2$.  Every vertex not in $H_2$ has degree at least $3n/4-(n/2-1)>n/4$ in $G'$ and every vertex in $H_2$ has degree at least $3n/4$ in $G'$, thus Lemma \ref{ChvatalDegree} implies that $G'$ has a Hamiltonian cycle, as before.

So let $X_1, \dots, X_k$ be all of the blue components in $G_2$ (some may be singletons), and suppose that $|X_i|\leq n/2$ for all $i\in [k]$.  
Consider the multipartite graph $G''\subseteq H_1$ consisting only of edges going between the $X_i$'s.  Observe that $G''$ satisfies the conditions of Lemma \ref{partitematching}:  for all $i\in [k]$ and for all $v\in X_i$, we have $\deg_{G''}(v)\geq 3n/4-(|X_i|-1)$.  Hence we obtain a perfect matching which is contained in $E(H_1)$ (i.e., consisting entirely of red edges).  
\end{proof}

Finally we state the lemma which allows us to turn the connected matching in the reduced graph into the cycle in the original graph. Some variant of this lemma, first introduced by \L uczak \cite{Lu}, has been utilized by many authors (\cite{BBGGS}, \cite{GyRSS1}, \cite{GyRSS2}, \cite{GyS}, \cite{LRS}).  See Lemma 2.2 in \cite{BLSSW} for the variant of \L uczak's lemma which is used to build the nearly spanning paths in each pair (in place of the much stronger blow-up lemma).  

\begin{lemma}\label{matchingcycle}
Let $0<\ep\ll d,\rho$ and let $\Gamma$ be an $(\ep, d)$-reduced graph of a $2$-colored graph $G$.  Assume that there is a monochromatic connected matching $\cM$ saturating at least $c|V(\Gamma)|$ vertices of $\Gamma$, for some positive constant $c$.  If $U\subseteq V(G)$ is the set of vertices spanned by the clusters in $\cM$, then there is a monochromatic cycle in $G$ covering at least $c(1-6\sqrt{\ep})n$ vertices of $U$.  Furthermore, if $R, S\subseteq U$ with $|R|, |S|\geq \rho n$, then there is a monochromatic path in $G$ covering at least $c(1-6\sqrt{\ep})n$ vertices of $U$ which has one endpoint in $R$ and the other endpoint in $S$.
\end{lemma}

When we apply Lemma \ref{matchingcycle}, there will be an existing path $P^*=v_1\dots v_k$ having the property that $R=N(v_1)\cap U$ and $S=N(v_k)\cap U$ with $|R|, |S|\geq \eta n\gg \ep n$ and thus we can find a path with one endpoint in $R$ and the other endpoint in $S$ giving a cycle which contains $P^*$ as a segment. We refer the reader to Lemma 3.5 in \cite{BBGGS} for more details.

\subsection{Proof of Theorem \ref{mainapprox}}

\begin{proof}
Let $\gamma>0$ be given and choose constants satisfying $\frac{1}{n_0}\ll \ep\ll d\ll \rho\ll \alpha\ll \eta\ll \gamma$.  Let $G$ be a $2$-colored graph on $n\geq n_0$ vertices with $\delta(G)\geq (3/4+\gamma)n$. Apply Lemma \ref{robuststructure} to $G$.  Either we directly obtain the desired monochromatic cycle partition in which case we are done, or we obtain an $(8\eta, 8\alpha)$-robust subgraph $H_1$ with $|H_1|\geq (1-\eta^{2/3})n$, or we obtain an $(8\eta, 8\alpha)$-nice partition consisting of $(8\eta, 8\alpha)$-robust subgraphs $H_1$ and $H_2$.  Either way, set $Z=V(G)\setminus V(H_1)$.  In the following paragraph, apply statements regarding $H_2$ only if $H_2$ has been defined according to the case we are in.

If $H_2$ exists (i.e. Lemma \ref{robustpartition}.(ii) holds), then without loss of generality, suppose $|H_1|\geq |H_2|$.  If $H_1$ is not $\alpha^4$-near bipartite, apply Lemma \ref{lemma:absorbing} to $H_1$ to get an absorbing path $P_1^*$ with $|P_1^*|\leq \rho n$.  By Observation \ref{slicerobust}, $H_2-P_1^*$ is $(4\eta, 4\alpha)$-robust. If $H_1$ is $\alpha^4$-near-bipartite, then by Observation \ref{bipartiterobust}, $H_1$ has a spanning bipartite subgraph $H_1[X_1,Y_1]$ with $|X_1|\leq |Y_1|$ which is $(4\eta, 4\alpha)$-robust and by Lemma \ref{robuststructure} and Definition \ref{nicepartition}, $H_2$ has the property that $H_2$ is not $\alpha^4$-near-bipartite, $|H_2|\geq (1/2+\eta^2)n$, and $|V(H_2)\cap Y_1|\geq \eta^{1/2}n$.  Apply Lemma \ref{lemma:absorbing} to $H_1[X_1,Y_1]$ to get an absorbing path $P_1^*$ with $|P_1^*|\leq \rho n$, furthermore let $S_1\subseteq X_1$ and  $T_1\subseteq Y_1\cap V(H_2)$ such that $|T_1|=\ceiling{6\rho^4 n}$ and $|S_1|=\floor{2\rho^4 n}$.  Note that the role of the sets $S_1$, $T_1$ is to ensure that in the case that $H_1$ is $\alpha^4$-near-bipartite, all of the vertices which cannot be absorbed into $P_1^*$ will be in $H_2$ (which is guaranteed to not be $\alpha^4$-near-bipartite) and thus can be absorbed by some $P_2^*$ in $H_2$.  The details will be made explicit below.  We split the proof into two main cases, essentially corresponding to the two main conclusions of Lemma \ref{robuststructure}.

\noindent
\textbf{Case 1} ($|H_1|\geq (1-\eta^{2/3})n$ and $H_1$ is not $\alpha^4$-near bipartite) (See Lemma \ref{robuststructure}.(i))  Note that by Lemma \ref{robuststructure}, the vertices in $Z$ have red degree less than $8\eta n$ to $V(H_1)$ and thus have blue degree at least $(3/4+\gamma/2)n$ to $V(H_1)$.  We further split into cases depending on whether or not $|Z|$ is large enough to apply regularity with $Z$ as an initial part of the partition or not.

\textbf{Case 1a} ($|Z|<3\rho n$)  If $|Z|>0$, use the fact that the vertices in $Z$ have large blue degree to greedily build a blue path $P_2'$ on exactly $2|Z|-1< 6\rho n$ vertices which contains all the vertices of $Z$, avoids all previously used vertices, and which has both endpoints in $Z$ (if $|Z|=1$, then $P_2'$ will consist of a single vertex). 

Let $G':=G-P_1^*-P_2'$ and let $H_1':=H_1-P_1^*-P_2'$; note that by Observation \ref{slicerobust}, $H_1'$ is $(2\eta, 2\alpha)$-robust.  Note that $\mathcal{Q}=\{V(H_1')\}$ forms a partition of $V(G')$ into sets of size at least $\rho^4 n$ and $\mathcal{Q}$ is clearly non-empty.

Apply Lemma \ref{2colordegreeform} to get a partition $\{V_0, V_1, \dots, V_{2k}\}$ of $G'$ respecting the partition $\mathcal{Q}$ and let $\Gamma$ be the $(\ep, d)$-reduced graph on $2k$ vertices as defined in Definition \ref{def:reduced}.  By Lemma \ref{connectedcomponent}, the graph $\cH_1$ induced by the clusters inside $V(H_1')$ is connected in $\Gamma_1$; note that $V(\cH_1) = V(\Gamma)$.  Let $\cH_2$ be the largest component in $\Gamma_2$.  By Lemma \ref{reduceddegree} and since each $\cH_i$ is maximal, we may apply Lemma \ref{pureconnectedmatching} to get a perfect matching $\cM$ which is contained in $E(\cH_1)\cup E(\cH_2)$; note that if $|\cH_2|\leq |\Gamma|/2$, then the matching is entirely red. In this case, we complete the path $P_2'$ into a cycle $C_2'$ by choosing a common neighbor of the endpoints.  Otherwise, the subgraph $\hat{H}_2$ in $G_2$ which contains the clusters from $\cH_2$ has order at least $(1/2-\gamma/2)n$ and thus the endpoints of $P_2'$ have at least $(1/4+\gamma)n$ neighbors in $\hat{H}_2$.  Thus we can apply Lemma \ref{matchingcycle}, to get cycles $C_1'$, $C_2'$ (containing $P_1^*$ and $P_2'$ respectively) covering all but at most $6\sqrt{\ep} n$ vertices of $G'$. Denote the leftover vertices of $G$ by $W$ and note that $V_0\subseteq W$.  Since we are in the case where $H_1$ is not $\alpha^4$-near-bipartite and all of the vertices from $W$ are contained in $H_1$, they can be absorbed into $P_1^*$.

\textbf{Case 1b} ($|Z|\geq 3\rho n$)  In this case there are too many vertices in $Z$ to deal with greedily, but enough so that we will be able to apply regularity with $Z$ as an initial part of the partition.  However, we still need to deal the leftover vertices from $Z$ in some way.  By Chernoff, there exists a set $R\subseteq V(H_1-P_1^*)$ of size $\rho^4n$ such that for all $u,v\in Z$, we have $|(N_2(u)\cap N_2(v))\cap R|\geq (1/2+\gamma/2)|R|$.  Set $G':=G-P_1^*-R$ and $H_1':=H_1-P_1^*-R$.  Note that $\mathcal{Q}=\{V(H_1'), Z\}$ partitions $V(G')$ into sets of size at least $\rho^4 n$ and is clearly non-empty.

Apply Lemma \ref{2colordegreeform} to get a partition $\{V_0, V_1, \dots, V_{2k}\}$ of $G'$ respecting  $\mathcal{Q}$ and let $\Gamma$ be the $(\ep, d)$-reduced graph on $2k$ vertices as defined in Definition \ref{def:reduced}.  By Lemma \ref{connectedcomponent}, the graph $\cH_1'$ induced by the clusters inside $V(H_1')$ is connected in $\Gamma_1$.  For $i\in [2]$ let $\cH_i$ be the largest component of color $i$ in $\Gamma_i$.  We have $|\cH_1|\geq 3|\Gamma|/4$, and $|\cH_2|\geq 3|\Gamma|/4$ as the clusters in $Z$ have blue degree greater than $3|\Gamma|/4$ (by Lemma \ref{reduceddegree}).  This in particular implies that $V(\cH_1')\subseteq V(\cH_1)$.  
By Lemma \ref{reduceddegree} and since each $\cH_i$ is maximal, we may apply Lemma \ref{pureconnectedmatching} to get a perfect matching $\cM$ which is contained in $E(\cH_1)\cup E(\cH_2)$.  Now we apply Lemma \ref{matchingcycle} to get a red cycle $C_1'$ containing $P_1^*$ and a blue path $P_2'$ with both endpoints in $Z$ covering all but at most $6\sqrt{\ep} n$ vertices of $G'$ which we denote by $W_0$ and note that $V_0\subseteq W_0$.  

Now we use $R$ to greedily complete our blue path $P_2'$ into a cycle $C_2$ which uses all of the remaining vertices from $Z$; note that this is possible since every pair of leftover vertices has at least $\gamma|R|/2\geq \gamma \rho^4 n/2\gg 6\sqrt{\ep} n$ common neighbors in $R$.  Finally, let $W$ be the remaining vertices from $W_0$ and $R$, which are all contained in $V(H_1)$ and note that $|W|\leq |R|+|W_0|\leq 6\sqrt{\ep} n+\rho^4n < \rho^3n$.  So all of the remaining vertices can be absorbed into $P_1^*$, thus completing the cycle partition.

\begin{figure}[ht]
\centering
\subfloat[Case 2a]{\label{5cases2}
\scalebox{.64}{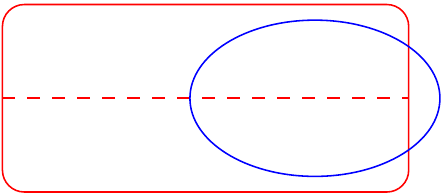}
}~~~~~~~~
\subfloat[Case 2b]{\label{5cases1}
\scalebox{.64}{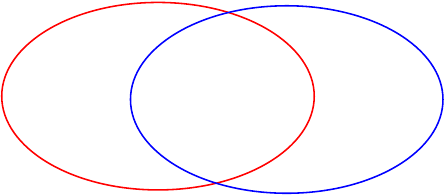~~~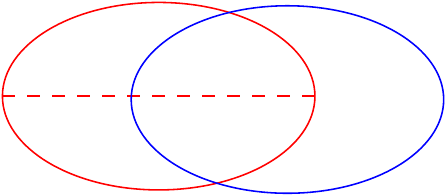~}
}

\caption[]{Case 2 in the proof of Theorem \ref{mainapprox}.  The dashed lines represent a near-bipartite component.}
\end{figure}

\textbf{Case 2} ($|H_1|< (1-\eta^{2/3})n$ or $H_1$ is $\alpha^4$-near bipartite) (See Lemma \ref{robuststructure}.(ii))  
The first subcase deals with when $H_1$ is near-bipartite and $|Z|$ is very small (Lemma \ref{robuststructure}.(ii.b)) and the second subcase deals with the case when $|Z|$ is not too small (Lemma \ref{robuststructure}.(ii.a,b)).

\textbf{Case 2a} ($|Z|<3\rho n$; i.e. $|H_1|> (1-3\rho)n$ and $H_1$ is $\alpha^4$-near bipartite) If $|Z|>0$, use the fact that the vertices in $Z$ have large blue degree to greedily build a blue path $P_2'$ on exactly $2|Z|-1< 6\rho n$ vertices which contains all the vertices of $Z$, avoids all previously used vertices, and which has both endpoints in $Z$ (if $|Z|=1$, then $P_2'$ will consist of a single vertex).  Finally, we find an absorbing path $P_2^*$ in $H_2-P_2'-P_1^*$.  Then we use the the connecting property of $H_2$ to connect one end of $P_2'$ to one end of $P_2^*$ forming a path $\hat{P}_2$ with one endpoint in $H_2$ and the other in $Z$. 

Let $G':=G-P_1^*-\hat{P}_2$ and for $i\in [2]$, let $H_i':=H_i-P_1^*-\hat{P}_2$; note that by Observation \ref{slicerobust}, $H_i'$ is $(2\eta, 2\alpha)$-robust.  Let $A:=V(H_1')\setminus V(H_2')$ and $B:=V(H_1')\cap V(H_2')$ (see Figure \ref{5cases2}).  If $|A|>\rho^4n$, then set $W':=\emptyset$ and $\mathcal{Q}=\{A,B\}$. If $|A|\leq \rho^4n$, then set $W':=A$ and $\mathcal{Q}=\{B\}$.  Finally set $G'':= G'-W'$ and $H_i'':=H_i'-W'$ for $i\in [2]$ and note that $H_i''$ is $(\eta, \alpha)$-robust by Observation \ref{slicerobust}.  Note that $\mathcal{Q}$ forms a partition of $G''$ into sets of size at least $\rho^4 n$ and $\mathcal{Q}$ is clearly non-empty.

Apply Lemma \ref{2colordegreeform} to get a partition $\{V_0, V_1, \dots, V_{2k}\}$ of $G''$ respecting the partition $\mathcal{Q}$ and let $\Gamma$ be the $(\ep, d)$-reduced graph on $2k$ vertices as defined in Definition \ref{def:reduced}.  By Lemma \ref{connectedcomponent}, the graph $\cH_i''$ induced by the clusters inside $V(H_i'')$ is connected in $\Gamma_i$ for $i\in [2]$.  For $i\in [2]$ let $\cH_i$ be the largest component of color $i$ in $\Gamma_i$.  So $V(\cH_1'') = V(\cH_1) = V(\Gamma)$.  We also have $|\cH_2|> |\Gamma|/2$ (see the discussion preceding Case 1), which in particular implies that and $V(\cH_2'')\subseteq V(\cH_2)$.  By Lemma \ref{reduceddegree} and since each $\cH_i$ is maximal, we may apply Lemma \ref{pureconnectedmatching} to get a perfect matching $\cM$ which is contained in $E(\cH_1)\cup E(\cH_2)$.
Note that the subgraph $\hat{H}_2$ in $G_2$ which contains the clusters from $\cH_2$ has order at least $n/2$ and contains all but at most $\ep n$ vertices of $H_2'$, and thus the endpoints of $\hat{P}_2$ have at least $\eta n$ neighbors in $\hat{H}_2$.  Thus we can apply Lemma \ref{matchingcycle}, to get cycles $C_1'$, $C_2'$ (containing $P_1^*$ and $\hat{P}_2$ respectively) covering all but at most $6\sqrt{\ep} n$ vertices of $G''$; denote the leftover vertices  of $G''$ by $W_0$ and note that $V_0\subseteq W_0$.  
  
Set $W:=W'\cup W_0$ and note that $|W|= |W'|+|W_0|\leq \rho^4n+6\sqrt{\ep} n< \rho^3n$.  Since $H_1$ is $\alpha^4$-near-bipartite, let $S_1'\subseteq S_1$ such that $|S_1'|=|W\cap Y_1|$ and let $T_1'\subseteq T_1$ such that $|T_1'|=|(S_1\setminus S_1')\cup (W\cap X_1)|$.  Since $|W\cap Y_1|+|T_1'|=|S_1'|+|(S_1\setminus S_1')\cup (W\cap X_1)|\leq 4\rho^4n$, these vertices can be absorbed into $P_1^*$ and the remaining vertices from $(T_1\setminus T_1')\cup (W\setminus V(H_1))$ can be absorbed into $P_2^*$.

\textbf{Case 2b} ($|Z|\geq 3\rho n$)  If $H_2$ is not $\alpha^4$-near-bipartite, apply Lemma \ref{lemma:absorbing} to $H_2-P_1^*$ to get an absorbing path $P_2^*$ with $|P_2^*|\leq \rho n$.  If $H_2$ is $\alpha^4$-near-bipartite, then by Observation \ref{bipartiterobust}, $H_2$ has a spanning bipartite subgraph $H_2[X_2,Y_2]$ with $|X_2|\leq |Y_2|$ which is $(4\eta, 4\alpha)$-robust and by Lemma \ref{robuststructure} and Definition \ref{nicepartition}, $H_1$ has the property that $H_1$ is not $\alpha^4$-near-bipartite and $|V(H_1)\cap Y_2|\geq \eta^{1/2}n$.  Apply Lemma \ref{lemma:absorbing} to $H_2[X_2,Y_2]$ to get an absorbing path $P_2^*$ with $|P_2^*|\leq \rho n$, furthermore let  $S_2\subseteq X_2$  and $T_2\subseteq Y_2\cap V(H_1)$ such that $|T_2|=\ceiling{6\rho^4 n}$ and $|S_2|=\floor{2\rho^4 n}$.

Now we proceed almost exactly as before. Let $G':=G-P_1^*-P_2^*$ and for $i\in [2]$, let $H_i':=H_i-P_1^*-P_2^*-S_i-T_i$ (where $S_i, T_i\neq \emptyset$ if and only if $H_i$ is $\alpha^4$-near-bipartite); note that by Observation \ref{slicerobust}, $H_i'$ is $(2\eta, 2\alpha)$-robust.  Let $A:=V(H_1')\setminus V(H_2')$, $B:=V(H_1')\cap V(H_2')$, and $C:=V(H_2')\setminus V(H_1')$ (see Figure \ref{5cases1}).  If $|A|>\rho^4n$, then set $W':=\emptyset$ and $\mathcal{Q}=\{A,B,C\}$. If $|A|\leq \rho^4n$, then set $W':=A$, $\mathcal{Q}=\{B,C\}$.  Finally set $G'':= G'-W'$ and $H_i'':=H_i'-W'$ for $i\in [2]$ and note that $H_i''$ is $(\eta, \alpha)$-robust by Observation \ref{slicerobust}.  Note that $\mathcal{Q}$ forms a partition of $G''$ into sets of size at least $\rho^4 n$ and $\mathcal{Q}$ is clearly non-empty.

Apply Lemma \ref{2colordegreeform} to get a partition $\{V_0, V_1, \dots, V_{2k}\}$ of $G''$ respecting the partition $\mathcal{Q}$ and let $\Gamma$ be the $(\ep, d)$-reduced graph on $2k$ vertices as defined in Definition \ref{def:reduced}.  By Lemma \ref{connectedcomponent}, the graph $\cH_i''$ induced by the clusters inside $V(H_i'')$ is connected in $\Gamma_i$ for $i\in [2]$.  For $i\in [2]$ let $\cH_i$ be the largest component of color $i$ in $\Gamma_i$. We have $|\cH_1|\geq 3|\Gamma|/4$ and $|\cH_2|\geq 3|\Gamma|/4$ as the clusters in $V(H_2'')\setminus V(H_1'')$ have degree greater than $3|\Gamma|/4$ (by Lemma \ref{reduceddegree}).  This in particular implies that $V(\cH_i'')\subseteq V(\cH_i)$ as $|V(\cH_i'')|>|\Gamma|/2$ for $i\in [2]$.  
By Lemma \ref{reduceddegree} and since each $\cH_i$ is maximal, we may apply Lemma \ref{pureconnectedmatching} to get a perfect matching $\cM$ which is contained in $E(\cH_1)\cup E(\cH_2)$.  Now we apply Lemma \ref{matchingcycle} to get cycles $C_1'$, $C_2'$ (containing $P_1^*$ and $P_2^*$ respectively) covering all but at most $6\sqrt{\ep} n$ vertices of $G''$ which we denote by $W_0$ and note that $V_0\subseteq W_0$.  

Set $W:=W'\cup W_0$ and note that $|W|= |W'|+|W_0|\leq \rho^4n+6\sqrt{\ep} n< \rho^3n$. If we are in the case where $H_1$ and $H_2$ are not $\alpha^4$-near-bipartite, the vertices from $W$ which are contained in $H_1$ can be absorbed into $P_1^*$ and the remaining vertices from $W$ which are contained in $H_2$ can be absorbed into $P_2^*$.  If $H_1$ is $\alpha^4$-near-bipartite, then let $S_1'\subseteq S_1$ such that $|S_1'|=|W\cap Y_1|$ and let $T_1'\subseteq T_1$ such that $|T_1'|=|(S_1\setminus S_1')\cup (W\cap X_1)|$.  Since $|W\cap Y_1|+|T_1'|=|S_1'|+|(S_1\setminus S_1')\cup (W\cap X_1)|\leq 4\rho^4n$, these vertices can be absorbed into $P_1^*$ and the remaining vertices from $(T_1\setminus T_1')\cup (W\setminus V(H_1))$ can be absorbed into $P_2^*$.   If $H_2$ is $\alpha^4$-near-bipartite, then let $S_2'\subseteq S_2$ such that $|S_2'|=|W\cap Y_2|$ and let $T_2'\subseteq T_2$ such that $|T_2'|=|(S_2\setminus S_2')\cup (W\cap X_2)|$.  Since $|W\cap Y_2|+|T_2'|=|S_2'|+|(S_2\setminus S_2')\cup (W\cap X_2)|\leq 4\rho^4n$, these vertices can be absorbed into $P_2^*$ and the remaining vertices from $(T_2\setminus T_2')\cup (W\setminus V(H_2))$ can be absorbed into $P_1^*$.   
\end{proof}

\section{Conclusion}\label{sec:conclusion}

After determining the robust structure of the graph, we show that regularity can be applied so that the reduced graph satisfies certain degree conditions which allow us to find a perfect matching.  Gy\'arf\'as, S\'ark\"ozy, and Szemer\'edi \cite{GySS} proved a stability version of Theorem \ref{GGthm} and their proof made use of the ``connected matching" approach, but they introduced a method which avoided the use of regularity.  It would be interesting to see if their method can be applied here to avoid the use of the regularity; this is part of the reason we proved Lemma \ref{robuststructure} without the use of regularity.

Erd\H{o}s, Gy\'arf\'as, and Pyber \cite{EGP} conjectured that every $r$-colored $K_n$ has a partition into at most $r$ monochromatic cycles.  This conjecture was recently disproved for $r\geq 3$ by Pokrovskiy \cite{P}, although his examples do have $r$ disjoint monochromatic cycles which together miss only one vertex.  Gy\'arf\'as, Ruszink\'o, S\'ark\"ozy, and Szemer\'edi proved that a monochromatic cycle partition can be found with at most $O(r\log r)$ cycles \cite{GyRSS1} and for $r=3$, proved that a partition can be found with at most $17$ cycles \cite{GyRSS2}.  It would be interesting to determine if a partition can be found with at most $4$ cycles for $r=3$, or even better, a partition with $4$ cycles having the extra condition that one of the cycles has order $1$.  We believe that the methods introduced here could provide an approach to this problem and this is part of the reason that Definition \ref{nicepartition} is stated for $r$ colors.

\section{Acknowledgements}

We greatly appreciate the work of the referees.  Their careful reading resulted in many improvements to the paper.

\end{document}